\documentclass[12pt, reqno]{amsart}
\usepackage{mathptmx, enumerate, amsmath, amsthm, amscd, amsfonts, amssymb, graphicx, color}\usepackage[bookmarksnumbered, colorlinks, plainpages]{hyperref}
\hypersetup{colorlinks=true,linkcolor=red, anchorcolor=green, citecolor=cyan, urlcolor=red, filecolor=magenta, pdftoolbar=true}
\input{mathrsfs.sty}
\usepackage{bbm}
\usepackage{amssymb}
\usepackage{epic}
\usepackage{ecltree}
\usepackage{tikz}
\usepackage{amssymb}
\usepackage{epic}
\usepackage{ecltree}
\usepackage{tikz}
\usepackage{enumerate}
\usepackage[mathscr]{eucal}
\usepackage{amsfonts}
\newtheorem{theorem}{Theorem}
\newtheorem{lemma}[theorem]{Lemma}
\newtheorem{proposition}[theorem]{Proposition}
\newtheorem{corollary}[theorem]{Corollary}
\newtheorem{problem}[theorem]{Problem}

\newtheorem{question}[theorem]{Question}

\theoremstyle{definition}
\newtheorem{definition}[theorem]{Definition}
\newtheorem{example}[theorem]{Example}
\newtheorem{remark}[theorem]{Remark}
\newtheorem{note}[theorem]{Note}
\newtheorem*{acknowledgement}{Acknowledgement}

\newcommand{\supp}{{\rm supp}\hskip0.02cm}
\newcommand{\1}{\mathbf{1}}

\newcommand{\tp}{{\rm TP}\hskip0.02cm}
\newcommand{\tc}{{\rm TC}\hskip0.02cm}
\def\lip{\hskip0.02cm{\rm Lip}\hskip0.01cm}

\newcommand{\orth}{\rm orth}

\newcommand{\BM}{\rm BM}
\newcommand{\gr}{\mathscr G}
\newcommand{\ba}{\mathcal B}
\newcommand{\x}{\mathscr X}
\newcommand{\rel}{\mathscr R}
\def\cay{\hskip0.02cm{\rm Cay}\hskip0.01cm}
\def\red{\mathbb{R}^E}
\def\rv{\mathbb{R}^V}
\newcommand{\ta}{{\rm tail}\hskip0.01cm}
\newcommand{\he}{{\rm head}\hskip0.01cm}
\newcommand{\rank}{{\rm rank}\hskip0.02cm}
\newcommand{\au}{{\rm Aut}\hskip0.02cm}
\newcommand{\cy}{{\rm Cyc}\hskip0.02cm}

\begin{document}
\title[Cycle spaces]{Cycle spaces: invariant projections and applications to transportation cost} 
\dedicatory{This paper is dedicated to Professor Per H. Enflo}




\author[S.\,J. Dilworth]{Stephen J. Dilworth}

\address{Department of Mathematics, University of South Carolina, Columbia,
SC 29208, USA}

\email{dilworth@math.sc.edu}


\author[D. Kutzarova]{Denka Kutzarova}

\address{Department of Mathematics, University of Illinois Urbana-Champaign,
Urbana, IL 61807, USA; Institute of Mathematics and Informatics, Bulgarian Academy of Sciences, Sofia, Bulgaria}

\email{denka@illinois.edu}


\author[M.\,I.~ Ostrovskii]{Mikhail I.~Ostrovskii}

\address{Department of Mathematics and Computer Science, St. John's
University, 8000 Utopia Parkway, Queens, NY 11439, USA}
  \email{ostrovsm@stjohns.edu}

\date{\today}



\begin{abstract} The paper starts with discussion of applications of cycle spaces to transportation cost. After a short survey of the known results
on cycle spaces, we turn to the study of minimal projections onto
cycle spaces in the corresponding $\ell_1$-spaces. This study is
naturally related to the study of invariant projections on the
cycle space, which, in turn, are determined by the properties of
representations of the automorphism group of the corresponding
graph. The main focus is on discrete tori and Hamming graphs.
\end{abstract}

\maketitle

\tableofcontents

\section{Introduction}

\subsection{Cycle and cut spaces}\label{S:CycCut}

Consider a connected graph $G$ with vertex set $V$ and edge set
$E$ (disconnected graphs can be analyzed componentwise). We
consider only {\it simple} graphs (that is, graphs without loops
and parallel edges). We follow the graph theory terminology of
\cite{BM08,Die17}.

In some cases, we consider \emph{weighted graphs}. This means that
there are positive numbers assigned to edges, and that these
numbers satisfy the following condition: the weight of the edge
$uv$ does not exceed the total weight of any path joining $u$ and
$v$. We endow the vertex set of such graphs with the shortest path
distance: the distance $d(u,v)=d_G(u,v)=$(the minimal total weight
of paths joining $u$ and $v$). For graphs with no weights we
assume that the weight (length) of an edge is equal to $1$.

In the other direction, each finite metric space $(X,d)$ can be
identified with an apposite weighted graph. This weighted graph
will be specified uniquely according to the procedure described
below and - following \cite{AFGZ21} - will be called the
\emph{canonical graph associated with the metric space}. The graph
is defined as follows. First, we consider a complete weighted
graph with vertices in $X$, and the weight of an edge $uv$
($u,v\in X$) defined as $d(u,v)$. After that we delete all edges
$uv$ for which there exists a vertex $w\notin\{u,v\}$ satisfying
$d(u,w)+d(w,v)=d(u,v)$. It is easy to see that, for a finite
metric space, the described procedure is a well defined and leads
to a uniquely determined by $(X,d)$ weighted graph, which we
denote $G(X, E)$, where $E=E(X)$ is the edge set of the canonical
graph, and $X$ serves as a vertex set of $G$.

As a result, the constructed in this way graph $G(X, E)$ possesses
the following  feature, which is crucial for its application in
the theory of the transportation cost spaces:
 the weighted graph distance of
$G(X, E)$ on $X$ coincides with the original metric on $X$. Note
that if $X$ is defined initially as a weighted graph with its
weighted graph distance, then the corresponding canonical graph
can be different because some edges can be dropped. However, if we
start with a simple unweighted graph endowed with the graph
distance, we will recover it as a canonical graph.

To proceed,  we  specify an orientation on the edge set $E$. This
means that we fix a \emph{direction} on the edges by selecting the
\emph{head} and \emph{tail} of for each edge. The orientation can
be chosen arbitrarily as the notions important for further
reasoning do not depend on the choice of this orientation. We call
this orientation a \emph{reference orientation}.

Notation: if $u$ is a tail and $v$ is a head of a directed edge
$e=uv$, we denote the directed edge either $\overrightarrow{uv}$
or $\overleftarrow{vu}$. We also write $u=\ta(e)=e^-$ and
$v=\he(e)=e^+$.

\begin{note} For some `well-organized' graphs there exist convenient choices of reference orientations.
\end{note}

The \emph{incidence matrix} $D$ of an oriented (= directed) graph
$G$ is defined as a matrix whose rows are labelled using vertices
of $G$, columns and labelled using edges of $G$ and the $ve$-entry
is given by
\[
d_{ve}=\begin{cases} 1, & \hbox{ if }v=e^+,\\
-1, & \hbox{ if }v=e^-,\\
0, & \hbox{ if $v$ is not incident to }e.\end{cases} \]

Denote by $\red$ and $\rv$ the space of real-valued functions on
the edge set and the vertex set, respectively. Interpreting
elements of $\red$ and $\rv$ as column vectors, we may regard $D$
as a matrix of a linear operator $D:\red\to \rv$. We also consider
the transpose matrix $D^T$ and the corresponding operator
$D^T:\rv\to \red$.\medskip

It is easy to describe $\ker D^T$. In fact, for $f\in \rv$ the
value of $D^T(f)\in \red$ at an edge $e$ is $f(e^+)-f(e^-)$,
therefore $f\in\ker D^T$ if and only if it has the same value at
the ends of each edge. Since we assumed that $G$ is connected,
this happens if and only if $f$ is constant on $V(G)$. Therefore
the ranks of the operators $D^T$ and $D$ are equal to
$|V|-1$.\medskip

Observe that
\begin{equation}\label{E:DF}(DF)(v)=\sum_{e,~
e^+=v}F(e)-\sum_{e,~e^-=v}F(e)\end{equation} for $F\in  \red$.

\begin{definition}\label{D:CycleCutSp}
We introduce the \emph{cycle space} (also called the \emph{flow
space} or \emph{cycle subspace}) of the graph $G$ as $Z=Z(G)=\ker
D\subseteq\red$. The name is chosen because, as we explain below,
$Z(G)$ is spanned by the set of signed indicator functions (see
\eqref{E:SignInd}) of all cycles in $G$. The orthogonal complement
of $Z$ in $\red$ (with respect to the standard inner product) is
called the {\it cut space} or {\it cut subspace} of $G$; we denote
it $B=B(G)$. This name is chosen because $B(G)$ is spanned by {\it
cuts}, that is, by the signed indicator functions of sets of the
edges joining a subset $U$ in $V(G)$ with its complement $\bar U$.
These edge sets are oriented in such a way that heads of all edges
are in $U$.
\end{definition}

Cycle, cut, and flow spaces are standard notions of the algebraic
graph theory; see \cite{Big93,Big97,GR01}.

Consider a cycle $C$ in $G$ (in a graph-theoretical sense) with
one of the two possible orientations satisfying the following
condition: each vertex of $C$ is a head of exactly one edge and a
tail of exactly one edge. We introduce the \emph{signed indicator
function} $\chi_C\in \red$ of the oriented cycle\index{signed
indicator function of an oriented cycle} $C$ (in an oriented graph
$G$) by
\begin{equation}\label{E:SignInd}
\chi_C(e)=\begin{cases} 1 & \hbox{ if }e\in C\hbox{ and its
orientations in $C$ and $G$ are the same}\\
-1 & \hbox{ if }e\in C\hbox{ but its orientations in $C$ and $G$
are different}
\\
0 & \hbox{ if }e\notin C.
\end{cases}\end{equation}

It is immediate from \eqref{E:DF} that $\chi_C\in\ker D$. To show
that signed indicator functions span $\ker D$ we observe that
$\dim(\ker D)=|E(G)|-\rank(D)=|E(G)|-|V(G)|+1$. To see that there
are $|E(G)|-|V(G)|+1$ linearly independent cycles we pick a
spanning tree in $G$ and consider all cycles formed by the
spanning tree and edges of $G$ which are not in the tree (there
are $|E(G)|-|V(G)|+1$  of them). It is easy to see that these
cycles are linearly independent

This computation also implies that the dimension of the cut space
in $G$ is $|E(G)|-(|E(G)|-|V(G)|+1)=|V(G)|-1$ (for a connected
graph $G$).

There is a very useful standard basis in $B(G)$: consider all
vertices of $G$ except one, and consider cuts joining them with
the rest of the graph (that is we use the definition of a cut
joining $U$ and $\bar U$ for $U=\{v\}$). Also, as for cycle
spaces, there exist bases for $B(G)$ corresponding to spanning
trees. Namely, for each spanning tree and each edge in it consider
the cut $(U,\bar U)$ consisting of vertices of two connected
components of the spanning tree with a removed edge.

\subsection{Transportation cost spaces}

Let $(X,d)$ be a finite  metric space (as is noted above $X$ can
be identified with the canonical weighted graph having $X$ as its
vertex set). Consider a real-valued finitely supported function
$f$ on $X$ with a zero sum, that is, $\sum_{v\in \supp f}f(v)=0$.
A natural and important interpretation of such a function, is
considering it as a \emph{transportation problem}: one needs to
transport certain product from locations where $f(v)>0$ to
locations where $f(v)<0$.

One can easily see that $f$ can be  represented as
\begin{equation}\label{E:TranspPlan} f=a_1(\1_{x_1}-\1_{y_1})+a_2(\1_{x_2}-\1_{y_2})+\dots+
a_n(\1_{x_n}-\1_{y_n}),\end{equation} where $a_i\ge 0$,
$x_i,y_i\in X$, and $\1_u(x)$ for $u\in X$ is the {\it indicator
function} of $u$, defined by
\[\1_u(x)=\begin{cases} 1 &\hbox{ if }x=u,\\ 0 &\hbox{ if }x\ne u.
\end{cases} \]

We call each such representation a \emph{transportation plan} for
$f$ because it can be interpreted as a plan of moving $a_i$ units
of the product from $x_i$ to $y_i$. The \emph{cost} of the
transportation plan \eqref{E:TranspPlan} is defined as
$\sum_{i=1}^n a_id(x_i,y_i)$.

We denote the real vector space of all transportation problems by
$\tp(X)$. We introduce the \emph{transportation cost norm} (or
just \emph{transportation cost}) $\|f\|_{\tc}$ of a transportation
problem $f$ as the infimum of costs of transportation plans
satisfying \eqref{E:TranspPlan}. Using the triangle inequality and
the fact that $f$ is finitely supported it is easy to see that the
infimum of costs of transportation plans for $f$ is attained. A
transportation plan for $f$ whose cost is equal to $\|f\|_\tc$ is
called an \emph{optimal transportation plan}.

\begin{definition}\label{D:TC}
The finite-dimensional normed space $(\tp(X),\|\cdot\|_\tc)$ is
called a \emph{transportation cost space} and is denoted by
$\tc(X)$.
\end{definition}

Transportation cost spaces are of interest in many areas and are
studied under many different names. We list some of them in the
alphabetical order: \emph{Arens-Eells space, earth mover distance,
Kantorovich-Rubinstein distance, Lipschitz-free space, Wasserstein
distance}. We prefer to use the term {\it transportation cost
space} since it makes the subject of this work instantly clear to
a wide circle of readers and it also reflects the historical
approach leading to these notions (see \cite{Kan42,KG49}).
Interested readers can find a review of the main definitions,
notions, facts, terminology and historical notes pertinent to the
subject in \cite[Section 1.6]{OO19}.

Some more definitions which we use. Let  $\mu$ and $\nu$ be
probability measures on $X$. Then  $\mu = \sum_{x \in X} p_x
\delta_x$ and $\nu = \sum_{x \in X} q_x \delta_x$, where $p_x, q_x
\ge 0$, $\sum_{x \in X} p_x= \sum_{x \in X} q_x = 1$, and
$\delta_x$ denotes a unit point mass located at $x$.  The
associated transportation problem $f(\mu, \nu)$ is defined by
$f(\mu,\nu)(x) = p_x - q_x$.  Let $\mathcal{P}_1(X)$ be the set of
all probability measures on $X$. The \emph{$1$-Wasserstein
distance} $d_{W_1}$ on $\mathcal{P}(X)$ is defined by
$$d_{W_1}(\mu,\nu) = \|f(\mu,\nu)\|_{\tc}.$$

By a \emph{pointed metric space} we mean a metric space $(X,d)$
with a \emph{base point}, denoted by $O$. For a pointed metric
space $X$ by $\lip_0(X)$ we denote the space of all Lipschitz
functions $f:X\to\mathbb{R}$ satisfying $f(O)=0$. It is not
difficult to check that $\lip_0(X)$ is a Banach space with respect
to the norm $\|f\|=\lip(f)$ ($\lip(f)$ is the Lipschitz constant
of $f$). As is well known $\tc(X)^*=\lip_0(X)$ (see e.g.
\cite[Section 10.2]{O}).

We use the standard terminology of Banach space theory \cite{BL00}
and the theory of metric embeddings \cite{O}.

\subsection{Applications of cycle spaces to the theory of transportation cost}
We introduce several norms on $\red$. Let $G=(V(G),E(G))=(V,E)$ be
a finite graph. Then $\ell_1(E)$, $\ell_2(E)$, and
$\ell_\infty(E)$ are defined as $\red$ with the norms
$\|f\|_1=\sum_{e\in E}|f(e)|$, $\|f\|_2=\left(\sum_{e\in
E}|f(e)|^2\right)^{\frac12}$, and $\|f\|_\infty=\max_{e\in
E}|f(e)|$, respectively. We also consider the inner product
$\langle f,g\rangle$ associated with $\|f\|_2$.

If $G=(V(G),E(G))=(V,E)$ is the canonical (weighted) graph for a
metric space $(X,d)$, with the weight of each edge $uv\in E(X)$
given as $d(u,v)$, we introduce also the norm
$\|f\|_{1,d}=\sum_{uv\in E}|f(uv)|d(u,v)$, and denote the
corresponding normed space $\ell_{1,d}(E)$.

The following important observations establish relations between
the cycle spaces and transportation cost spaces.

\begin{enumerate}[{\bf (1)}]

\item The space $\tc(X)$ is isometric to the quotient space
$\ell_{1,d}(E)/Z(G)$. In the case of an unweighted graph $G$, the
space $\tc(G)$ is isometric to $\ell_1(E)/Z(G)$. The result was
proved in the latter case in \cite[Proposition 10.10]{O}. It was
generalized to the weighted case in \cite{OO20}. Different proofs
were obtained in \cite{AFGZ21,DKO2}. We recommend the readers to
check the simple transparent proof in \cite[Proposition
1.5]{OO22}.

\item Each isometry of $\tc(X)$ and $\tc(Y)$ is induced by a
cycle-preserving bijection of $E(X)$ and $E(Y)$ \cite{AFGZ21}.
This result allowed the authors of \cite{AFGZ21} to use Whitney's
result \cite{Whi33} on cycle-preserving bijections and classify
isometries of transportation cost spaces on finite metric spaces.
In particular, this result implies that if the canonical graphs
associated with $X$ and $Y$ are $3$-connected, then each isometry
between $\tc(X)$ and $\tc(Y)$ is induced by an isomorphism of the
corresponding graphs. (In the other direction the result holds if
the graphs are unweighted, or the weights are distributed
according to one of the isomorphisms. See a similar result for
Sobolev spaces on graphs in \cite{Ost07}.)

\end{enumerate}

The result {\bf (1)} shows that for a better understanding of the
structure of $\tc(X)$ it is crucial to understand the structure of
$Z(G)$ (as a subspace of $\ell_{1,d}(G)$) and its position in
$\ell_{1,d}(G)$.

\subsection{Old and new results and observations on cycle
spaces}\label{S:OldNew} In this section and in the rest of the
paper we restrict our attention to unweighted graphs.

\begin{enumerate}[{\rm (A)}]

\item It seems that the oldest result which can be regarded as a
result about cycle spaces is the following result of Erd\H{o}s and
P\'osa \cite{EP62} (reprinted in \cite{Erd73}): The cycle space of
dimension $d$ contains at least roughly $d/\log d$ edge-disjointly
supported vectors, and thus a $1$-complemented subspace in
$\ell_1(E)$ of the corresponding dimension.

\item The well-known observation about quotient spaces (see
\cite[Fact 4.4]{DKO}) implies that in order to show that $\tc(G)$
is ``far'' from $\ell_1^n$ with respect to the Banach-Mazur
distance, it suffices to show that the relative projection
constant $\lambda(Z(G),\ell_1(E))$ is large. This inspired the
work of \cite{DKO,DKO2} on estimates for $\lambda(Z(G),\ell_1(E))$
from below.

\item \label{I:Inv} Another tool which was systematically used in
this context, was the observation which goes back at least to
Gr\"unbaum \cite{Gru60} (see also later papers
\cite{Rud62,And78}). The observation states that the projection of
minimal norm onto a given subspace (in cases when existence is
obvious) can be found among so-called invariant projections onto
the subspace. We mean the following situation. Let $Y_0$ be a
subspace of a finite-dimensional Banach space $Y_1$ (there exist
infinite-dimensional generalizations, but we do not need them in
this paper). Let $\Omega$ be a finite group of isometries of $Y_1$
leaving invariant the subspace $Y_0$. A projection $P:Y_1\to Y_0$
is called \emph{invariant} if $P\omega=\omega P$ for every
$\omega\in\Omega$. It is easy to construct an invariant
projections starting with any projection $P_0:Y_1\to Y_0$. Namely,
one can easily check that
\begin{equation}\label{E:InvProj}
P_{\rm
inv}:=\frac1{|\Omega|}\sum_{\omega\in\Omega}\omega^{-1}P\omega
\end{equation}
is an invariant projection onto $Y_0$ and that $\|P_{\rm inv}\|\le
\|P_0\|$.

If we start by choosing a projection $P$ with the minimal possible
norm, then $\|P_{\rm inv}\|=\|P\|$. This leads us to a conclusion:
To find/estimate the relative projection constant
$\lambda(Y_0,Y_1):=\inf\{\|P\|:~P:Y_1\to Y_0$ is a projection$\}$,
it suffices to compute/estimate norms of invariant projections. In
this connection it becomes important to understand:

\begin{problem}\label{P:UnInvProj} In what cases is an invariant projection
$P:\ell_1(E)\to Z(G)$ unique?
\end{problem}

In cases where the invariant projection $P:\ell_1(E)\to Z(G)$ is
not unique, it is important to understand the structure of the set
of invariant projections.\medskip

\item We present a condition for uniqueness of invariant
projections in terms of group representations (Theorem
\ref{T:InvProjEqSub} and Corollary \ref{C:InvProjEqSub} below).
Our sources for basic facts of representation theory are
\cite{Ser77,Sim96}.
\end{enumerate}

We introduce the group $\cy(G)$ of isometries of $\ell_1(E)$ for
which $Z(G)$ is an invariant subspace. If the graph $G$ is
$3$-connected, then, by the mentioned in {\bf (2)} result of
Whitney \cite{Whi33}, the group $\cy(G)$ is canonically isomorphic
to the group $\au(G)$ of automorphisms of the graph $G$.

Since isometries of $\ell_1(E)$ are signed permutations of the
unit vector basis, they can also be regarded as orthogonal
operators on $\ell_2(E)$. This observation implies that not only
$Z(G)$, but also $B(G)$ is an invariant subspace of $\cy(G)$ and
$\au(G)$.

Thus, we have two natural representations of the groups $\cy(G)$
and $\au(G)$: one on $Z(G)$ and the other on $B(G)$. Both $Z(G)$
and $B(G)$ are regarded as subspaces of $\ell_2(E)$.

Note that representation theory is mostly devoted to unitary
representations in Hilbert spaces over complex field. However,
some of the results of this theory continue to hold for orthogonal
representations in Hilbert spaces over $\mathbb{R}$. This aspect
of the representation theory is presented in \cite{Sim96}.

\begin{theorem}\label{T:InvProjEqSub} The invariant projection $P:\ell_2(E)\to Z(G)$ is
unique if and only if the natural representations of $\cy(G)$ on
$Z(G)$ and $B(G)$ do not have equivalent restrictions to non-zero
subspaces.
\end{theorem}

Combining this theorem with the mentioned in {\bf (2)} result of
Whitney \cite{Whi33}, we get the following corollary.

\begin{corollary}\label{C:InvProjEqSub} For a $3$-connected graph $G$, the invariant projection $P:\ell_2(E)\to Z(G)$ is
unique if and only if the natural representations of $\au(G)$ on
$Z(G)$ and $B(G)$ do not have equivalent restrictions to non-zero
subspaces.
\end{corollary}

\begin{proof}[Proof of Theorem \ref{T:InvProjEqSub}] Observe that the orthogonal projection $P_{\rm orth}:\ell_2(E)\to Z(G)$
is invariant. The best way to see it is to write decompositions of
$P_{\rm orth}$ and an arbitrary $A\in\cy (G)$ with respect to the
orthogonal decomposition $\ell_2(E)=Z(G)\oplus B(G)$. In fact,
invariance of $Z(G)$ and $B(G)$ for any $A\in\cy(G)$, shows that
\[ A=\begin{pmatrix} A_Z & 0\\
0 & A_B\end{pmatrix},\] where $A_Z$ and $A_B$ are the
corresponding restrictions. Also, we have
\[P_{\rm orth}=\begin{pmatrix} I_{Z(G)} & 0\\
0 & 0\end{pmatrix},\] where $I_{Z(G)}$ is the identity operator on
$Z(G)$. The validity of $P_{\rm orth}A=AP_{\rm orth}$ is
immediate.

Observe that any projection of $\ell_2(E)$ onto $Z(G)$ is of the
form
\[P=\begin{pmatrix} I_{Z(G)} & Q\\
0 & 0\end{pmatrix},\] where $Q$ is some linear operator,
$Q:B(G)\to Z(G)$.

It is easy to see that $PA=AP$ is equivalent to
\begin{equation}\label{E:Qintertw} QA_B=A_ZQ.\end{equation}

We show that if $Q\ne 0$, then \eqref{E:Qintertw} implies that
$A_Z$ and $A_B$ have restrictions to non-zero subspaces which are
equivalent representations.

Note, that the observation at the beginning of \cite[Appendix to
III.5, p~50]{Sim96} shows that in the real case one can also
decompose an orthogonal representation into an orthogonal sum of
irreducible over the reals representations. Therefore, the
representations $A_Z$ and $A_B$ of $\cy(G)$ can be decomposed into
irreducible over the reals representations. Let
\[B(G)=V_1\oplus V_2\oplus \dots\oplus V_k\]
be the corresponding decomposition of $B(Z)$. If $Q\ne 0$, then
there exists $i\in\{1,\dots,k\}$ such that $Q|_{V_i}\ne 0$.
\medskip

Since $V_i$ is an invariant subspace for all restrictions $A_B$ of
$A\in\cy(G)$, the formula \eqref{E:Qintertw} implies that $Q(V_i)$
is an invariant subspace for $A_Z$, $A\in\cy(G)$.

Then the usual proof of the first part of Schur's lemma (as in
\cite{Ser77}) shows that the restriction of $A_B$ to $V_i$ and the
restriction of $A_Z$ to $Q(V_i)$ are equivalent representations.

To complete the proof it suffices to show that if $A_Z$ and $A_B$
have equivalent restrictions, then there exists $Q\ne 0$
satisfying \eqref{E:Qintertw}.

In fact, let the equivalent restrictions be to the subspaces
$V\subset B(G)$ and $U\subset Z(G)$. Let $R: V\to U$ be the
isomorphism establishing the equivalence, and let $P_V:B(G)\to V$
be the orthogonal projection. It is easy to see that $Q:B(G)\to
Z(G)$ given by $Q=RP_V$ is the desired operator.
\end{proof}

\subsection{A general problem on representations of automorphism groups}

Corollary \ref{C:InvProjEqSub} makes the following problem very
important. This problem can be regarded as a part of the general
problem of descriptions of representations of the groups $\au(G)$
which are among the most important groups in algebraic
combinatorics; see \cite{Big93, GR01, LS16}.

\begin{problem}\label{P:GenProblem} Given a graph $G$, whether the restrictions of the natural representation of $\au(G)$ on $\ell_2(E)$ to $Z(G)$ and
$B(G)$ have equivalent restrictions to non-zero subspaces?
\end{problem}

Of course, the same problem is of importance also for groups
$\cy(G)$, but these groups are used much less frequently.

We have to mention that we did not find any publications
explicitly stating any results on Problem \ref{P:GenProblem}.

\subsection{Relations between the norm of minimal projection and
other geometric parameters}

Henceforth $\|T\|_1$ (resp.\ $\|T\|_\infty$) denotes  the operator
norm of a linear operator $T \colon \ell_1(E) \rightarrow
\ell_1(E)$ (resp.\ $T \colon \ell_\infty(E) \rightarrow
\ell_\infty(E)$).

The comments in Section \ref{S:OldNew}\ \eqref{I:Inv} show that
the collection of invariant projections onto $Z(G)$ is non-empty
as the orthogonal projection $P_{\orth}$ is invariant. Also they
imply that there exist invariant projections of minimal possible
norm.

Minimal  projections onto $Z(G)$ yield information about the
projection constants of $\lip_0(G)$ and the $L_1$-distortion of
$(\mathcal{P}(G), d_{W_1})$. Let us recall the relevant
definitions.

Let $Y$ be a finite-dimensional normed space and let $Y_1$ be any subspace of $\ell_\infty$  that is isometrically isomorphic to $Y$. Recall that the \textit{projection constant} of $Y$, denoted $\lambda(Y)$, is defined by
$$ \lambda(Y) = \inf \{\|P\| \colon P \colon \ell_\infty \rightarrow \ell_\infty\text{ is a projection with range $Y_1$}\}.$$
(It is known that $\lambda(Y)$ is independent of the choice of
$Y_1$.)

Let $(X,d)$ and $(Y,\rho)$ be metric spaces. An injective map $F:
X\to Y$ is called a {\it bilipschitz embedding} if there exist
constants $C_1,C_2>0$ so that for all $u,v\in X$
\[C_1d(u,v)\le \rho(F(u),F(v))\le C_2d(u,v).\]
The {\it distortion} of $F$ is defined as
$\lip(F)\cdot\lip(F^{-1}|_{F(X)})$, where $\lip(\cdot)$ denotes
the Lipschitz constant.

The {\it $L_1$-distortion} $c_1(X)$ is  defined as the infimum of
distortions of all bilipschitz embeddings of $(X,d)$ into any
space $L_1(\Omega,\Sigma,\mu)$.

Differentiability of Lipschitz maps (see \cite[Theorem 7.9]{BL00})
implies that if $X$ is a finite-dimensional normed space, then
$c_1(X)$ is equal to the infimum of distortions of all linear
embeddings of $X$ into $L_1[0,1]$.

The \textit{Banach-Mazur distance} between   normed spaces $Y_1$
and $Y_2$ of the same finite  dimension is defined by
$$d_{\BM}(Y_1,Y_2) = \inf \{\|T\|\cdot\|T^{-1}\| \colon T\colon Y_1 \rightarrow Y_2 \text{\, is a linear isomorphism}\}.$$

The connection between $L_1$-distortion  and projection constants on the one hand and minimal  projections onto $Z(G)$  on the other is summarized in the following essentially known result.

\begin{proposition} \label{prop: minimalnormprops} Let $P_{\min}$ be a minimal  projection of $\ell_1(E)$ onto $Z(G)$. \begin{itemize}
\item[(i)]
$$ c_1(\tc(G)) \le \|I - P_{\min}\|_1$$ \item[(ii)] $$c_1((\mathcal{P}(G),d_{W_1}))  \le \|I - P_{\min}\|_1.$$
\item[(iii)]
$$\lambda(\lip_0(G)) \ge \|P_{\min}\|_1 - 1.$$ \item[(iv)] For $N = \operatorname{dim}(\tc(G))=|V(G)|-1$, $$d_{\BM}(\tc(G), \ell_1^N) \ge \|P_{\min}\|_1 -1.$$
 \end{itemize}
\end{proposition} \begin{proof} (i) Let $Q$ be any projection from $\ell_1(E)$ onto   $Z(G)$. For all  $x\in \operatorname{ker}(Q)$ and $y \in Z(G)$, we have
$$ \|x\|_1 = \|(I - Q)(x+y)\|_1 \le \|I -Q\|_1 \|x+y\|_1.$$ Hence, identifying $\tc(G)$ with $\ell_1(E)/Z(G)$, we get
$$ \|x + Z(G)\|_{\tc} \le \|x\|_1  \le \|I -Q\|_1 \|x+Z(G)\|_{\ell_1(E)/Z(G)} =\|I-Q\|_1 \|x + Z(G)\|_{\tc}.$$
Setting $Q = P_{\min}$, the linear mapping $T \colon \ell_1(E)/Z(G) \rightarrow  \operatorname{ker}(Q)$ given by $x + Z(G) \mapsto x$ yields
(i).

It was proved in \cite{NS} that $c_1((\mathcal{P}(G),d_{W_1}))=
c_1(\tc(G))$,  so (i) and  (ii) are equivalent.

 (iii) Note that $\operatorname{Lip}_0(G)=\tc(G)^*$ is isometrically isomorphic to $(B(G), \|\cdot\|_\infty) \subset l_\infty(E)$ by
$\tc(G)=\ell_1(E)/Z(G)$,  since $B(G) = Z(G)^\perp$. Now suppose
$Q$ is any projection from  $\ell_\infty(E)$ onto $B(G)$. Then $I-
Q^*$ is a projection from $\ell_1(E)$ onto $Z(G)$. So
$$\|Q\|_\infty \ge \|I - Q\|_\infty - 1 = \|I-Q^*\|_1-1 \ge \|P_{\text{min}}\|_1-1, $$ and hence
$$\lambda(\operatorname{Lip}_0(G))  = \inf \|Q\|_\infty \ge  \|P_{\text{min}}\|_1-1.$$

(iv)  By duality, \[ d_{\BM}(\tc(G), \ell_1^N)  =
d_{\BM}(\lip_0(G), \ell_\infty^N) \ge \lambda(\lip_0(G))
\ge\|P_{\min}\|_1 - 1. \qedhere\]
\end{proof}

\section{Invariant and minimal projections onto cycle spaces for different families
of graphs}

\subsection{Summary of results}

This article is a continuation of \cite{DKO} and \cite{DKO2} in
which invariant projections for the diamond, Laakso, and somewhat
more general recursive families of graphs were investigated, with
applications to the corresponding transportation cost spaces.

At this point, the established relation between the uniqueness of
invariant projections and properties of representations of
automorphism groups (Theorem \ref{T:InvProjEqSub} and Corollary
\ref{C:InvProjEqSub}) does not play a significant role in this
study. Increase of this role requires the development of the
representation theory of $\au(G)$ in the corresponding direction
(we have not found suitable results in the literature).

We have only one example of such application (see Section
\ref{S:GRR}). On the other hand, in view of Theorem
\ref{T:InvProjEqSub}, our results on invariant projections
immediately imply the corresponding consequences about
representations of $\au(G)$ for $G$ in question.

In Section \ref{S:CayBasCyc} we review definitions of Cayley
graphs, generators and relators of groups, and introduce a useful
system of cycles in Cayley graphs.

Section \ref{S:GRR} establishes non-uniqueness of invariant
projections for Cayley graphs of a group $\Gamma$ for which there
are no automorphisms except multiplications by elements of
$\Gamma$.

In Section~\ref{sec: torus}  we characterize invariant projections
on the $2$-dimensional torus $\mathbb{T}_n:=\mathbb{Z}_n\times
\mathbb{Z}_n$ and compute $\|P_{\min}\|_1$ for $n\le 6$. In
Section~\ref{sec: higherdimensions}  we extend the uniqueness
results of the previous section  to tori in higher dimensions and
to Hamming graphs. In the final section  we show that  the
orthogonal projection $P_{\rm{orth}}$  is minimal and that
$\|I-P_{\rm{orth}}\|_1 = \dfrac{n+1}{2}$ for the Hamming cube
$\mathbb{Z}_2^n$.

\subsection{Remarks on Cayley graphs}\label{S:CayBasCyc} The graphs which we consider are the Cayley graphs of finite groups.  Let us recall the definition. Relevant material of Group Theory
can be found in the introductory chapters of \cite{Rob96}, on
Cayley graphs - in \cite[Chapter 3]{GR01}.
\medskip

Let $\gr$ be a finite group with the identity element $1$ or
$1_\gr$ if the binary operation is multiplication, and $0$ or
$0_\gr$ if the binary operation is addition.

\begin{note}\label{N:MultAdd} We shall use multiplicative notation for general
groups. On the other hand, for $\mathbb{Z}_n$, considered as a
group with respect to addition, we use additive notation (to avoid
the confusion with the multiplication in the ring $\mathbb{Z}_n$).
\end{note}

We say that a subset $\x=\{x_i\}_{i=1}^k$ {\it generates} $\gr$ or
is a {\it generating set} for $\gr$ if every element of $\gr$ can
be written as a (finite) product of $\{x_i\}_{i=1}^k\bigcup
\{x_i^{-1}\}_{i=1}^k$. Such product is also called a {\it word}
formed from $\x$ and $\x^{-1}:=\{x^{-1}:~x\in\x\}$.
\medskip

A finite group $\gr$ with finite generating set $\x$ can be
described in terms of {\it relations} between elements of
$\x\cup\x^{-1}$ or {\it relators}. The set of all relators is
denoted $\rel$.

\begin{example} \label{ex: cayley} Group $\mathbb{Z}^2_n$, in additive notation,
has generating set $\{x,y\}$ and

\begin{enumerate}[\bf (1)]

\item relations $x+y=y+x$, $\underbrace{x+\dots+x}_{n~{\rm
times}}=0$, $\underbrace{y+\dots+y}_{n~{\rm times}}=0$

\item or relators $-x-y+x+y$, $\underbrace{x+\dots+x}_{n~{\rm
times}}$, $\underbrace{y+\dots+y}_{n~{\rm times}}$.
\end{enumerate}
\end{example}

So {\it relations} are equalities between words and {\it relators}
are words which have to be equal to the group identity element.

Let $\x$ be a generating subset of a group $\gr$ satisfying
$\x^{-1}=\x$. The \emph{right-invariant Cayley graph}
$\cay(\gr,\x)$ of $\gr$ corresponding to the generating set $\x$
is the graph whose vertices are elements of $\gr$ and elements
$g,h\in \gr$ are connected by an edge if and only if $hg^{-1}\in
\x$. The graph distance of this graph is called the \emph{word
metric} because the distance between group elements $g$ and $h$ is
the shortest representation of $hg^{-1}$ as a word in terms of
elements of $\x$.

We use the term {\it right-invariant} because multiplication of
all elements of $\gr$ on the right with a fixed element of $\gr$
is an isomorphism of the graph.

Similarly, one can define a \emph{left-invariant Cayley graph}. It
is different, but has similar properties, see \cite[Remark
5.10]{O}.

The following proposition will play an important role in our
results.

\begin{proposition}\label{P:GenCase}  The set of all signed indicator functions of directed
cycles of the form
\[y,g_ny,g_{n-1}g_ny,\dots,g_2\dots g_ny, g_1\dots
g_ny=y,\] where $y\in\gr$ and $r=g_1\dots g_n\in \rel$,
$g_1,\dots,g_n\in \x$,  spans the cycle space of $\cay(\gr,\x)$.
\end{proposition}

\begin{proof} It is easy to
see that if we describe a group $\gr$ in terms of generators $\x$
and relators $\rel$, then the set of all words which are equal to
$1$ is $\gr$ can be described as so-called \emph{normal closure}
of $\rel$, that is, the set of all finite products of the form
$w^{-1}r^{\pm 1}w$, where $w$ is any word and $r\in\rel$. In fact,
the images of all other words in the quotient of the free group
generated by $\x$ over the normal closure are not equal to the
identity.

Therefore, any word which represents a closed walk is obtained by
reducing the word of the form:

\begin{equation}\label{E:ClWalk} w_1^{-1}r_1^{\pm 1}w_1w_2^{-1}r_2^{\pm 1}w_2\dots
w_p^{-1}r_p^{\pm 1}w_p;\end{equation} {\it reducing} means
deleting products of the forms $x_ix_i^{-1}$ (see \cite[Sections
2.1-2.2]{Rob96} if you never studied this).

This implies that each cycle in $\cay(\gr,\x)$ is represented by
the same vector on the edge set $E(\cay(\gr,\x))$ as the closed
walk \eqref{E:ClWalk}, and this walk, because of cancellation of
terms corresponding to $w_i^{\pm1}$ is represented by  a linear
combination of images in $\ell_1(E(\cay(\gr,\x)))$ of cycles
corresponding to $r_i$ (note that the cycle corresponding to
$r_i^{-1}$ is just the negation of the cycle corresponding to
$r_i$).
\end{proof}

\subsection{Groups with graphical regular
representation}\label{S:GRR}

Godsil \cite{God81} proved the following result: Each non-solvable
finite group (see \cite[Chapter 5]{Rob96} for the definition and
basic facts on solvable groups) has a  subset $\x =\x^{-1}$ such
that the only automorphisms of the Cayley graph
$G:=\cay(\Gamma,\x)$ are shifts by elements of $\Gamma$, so
$\au(G)=\Gamma$.

\begin{theorem} For this $G$ there is no uniqueness of invariant projection
onto $Z(G)$.
\end{theorem}

The proof of this theorem is based on two lemmas.

\begin{lemma}\label{L:RepCut} The representation of $\Gamma=\au(G)$
on the cut space $B(G)$ contains all irreducible representation of
$\Gamma$ except the unit representation.
\end{lemma}

\begin{proof} In fact, consider the natural basis $\{x_g\}_{g\ne 1_\Gamma}$ in
$B(G)$ introduced at the end of Section \ref{S:CycCut}. If we
write the matrices of the representation of $\Gamma$ in this
basis, we get that the character has value $n-1$ (where
$n=|\Gamma|$) on $1_\Gamma$ and value $-1$ on all other elements.
So it is the difference of the characters of the regular
representation and the unit representation. Therefore (by general
theory, see \cite{Ser77,Sim96}) the representation of $\Gamma$ on
$B(G)$ contains all irreducible representations except the unit
representation.
\end{proof}

\begin{lemma}\label{L:RepCycle} The representation of $\Gamma=\au(G)$
on the cycle space $Z(G)$ contains irreducible representation of
$\Gamma$ which are not equivalent to the unit representation.
\end{lemma}

\begin{proof} It suffices to prove that there is a
vector $x\in Z(G)$ which is mapped by one of $g\in\Gamma$ onto a
vector which is linearly independent with $x$.

So we need to exhibit a cycle which is not mapped onto itself by
some element of the group. To see the existence of such cycle,
observe that only the cycle joining elements of a cyclic group in
a cyclic order is mapped onto itself by all elements of the group.
There are many reasons for which our situation is different.
\end{proof}

\subsection{Two-dimensional torus}
\label{sec: torus}

\subsubsection{Notation}

  Let  $n \ge 3$. The torus $\mathbb{T}_n = \mathbb{Z}_n \times \mathbb{Z}_n$ viewed as a graph  consists of vertices $(i,j)$ with $0 \le i,j < n$,  and edge set $E_n$ consisting of    edges
 oriented as follows :  $(i,j) \rightarrow (i+1,j)$ (horizontal left to right) and $(i,j) \rightarrow (i,j+1)$ (vertical bottom to top) with addition mod $n$. The corresponding ``row'' and  ``column''  edge vectors which form a basis of  the edge space  are denoted $e^R_{i,j}$ and $e^C_{i,j}$.
Let $\ell_1(E_n)$, $\ell_2(E_n)$, and $\ell_\infty(E_n)$ denote
the edge space  equipped with the usual $\ell_1$, $\ell_2$, and
$\ell_\infty$ norms with the edge vectors as the standard  unit
vector  basis. Let $Z_n=Z(\mathbb{T}_n)$ and $B_n=B(\mathbb{T}_n)$
denote the cycle and cut spaces.

Let $\Gamma_n= \Gamma(\mathbb{T}_n)$ be the group of linear isometries of $\ell_1(E_n)$  for which $Z_n$ is invariant.

Let $H_n = H(\mathbb{T}_n)$ be the automorphism group of $\mathbb{T}_n$. As mentioned above,  each $g \in H_n$ induces a
$\hat{g} \in \Gamma_n$. Let $U_n$ be the subgroup of $H_n$ generated by $\alpha$ and $\beta$, defined as follows:
$$ \alpha((i,j)) = (-i, j);\qquad \beta((i,j)) = (j,i).$$
Let $V_n$ be the subgroup of $H_n$ consisting of translations $T_{(a,b)}$  ($(a,b) \in \mathbb{T}_n$) defined by
$$T_{(a,b)}(i,j) = (a+i, b+j).$$

 A projection $P$ from $\ell_1(E_n)$ onto
 $Z_n$ or $B_n$  is \textit{$\Gamma_n$-invariant} if
$Pg = gP$ for all  $g \in \Gamma_n$.

 Corresponding to each $(i,j) \in \mathbb{T}^n$ there are
 a  ``small square''
cycle vector
$$ S((i,j)) :=  e^R_{i,j}  + e^C_{i+1,j} - e^R_{i,j+1} - e^C_{i,j}$$ and a  ``cross'' cut vector corresponding to the cut $(\{(i,j)\}, \mathbb{T}^n \setminus \{(i,j)\})$ $$ X((i,j)) :=  e^R_{i,j}   +e^C_{i,j} -  e^C_{i-1,j} - e^R_{i,j-1} .$$
We also introduce  ``row'' and  ``column'' cycle vectors
$$ R := \sum_{j=0}^{n-1} e^R_{j,0} \quad\text{and}\quad K:= \sum_{j=0}^{n-1} e^C_{0,j}.$$

\subsubsection{Invariant projections onto $Z_n$}
\begin{lemma} \label{lem: bases}   $$\ba_n^1 =\{ S((i,j)) \colon (i,j) \ne
(0,0)\} \cup \{R,K\}$$ and
$$\ba_n^2=\{ X(
(i,j))\colon (i,j) \ne (0,0)\}$$  are algebraic bases for $Z_n$
and $B_n$ respectively. \end{lemma}
\begin{proof} Clearly, $\ba_n^1$ is a linearly independent subset of $ Z_n$ and  $|\ba_n^1| = n^2 + 1 = \operatorname{dim}(Z_n)$.  Similarly, $\ba_n^2$ is a linearly independent subset of $B_n$ and $|\ba_n^2| = n^2 -1 = \operatorname{dim}(B_n)$. \end{proof}

\begin{lemma} \label{lem: induced} Let $n\ge 3$ and  $\theta \in \Gamma_n$. Then either $\theta$ or $-\theta$ is induced by a graph automorphism $\phi$.
\end{lemma}

\begin{proof} It is easy to see that each $\theta\in\Gamma_n$ induces an orthogonal operator on $\ell_2^(E)$. Therefore $\theta$ leaves the cut space $B_n$ invariant.
Therefore  each $\theta(X((i,j)))$ is a sum of four signed edge
vectors and is orthogonal to each $S((a,b))$. It is easily seen
that this implies that $\theta(X((i,j))) = \varepsilon((i,j))
X(\phi((i,j)))$, where $\varepsilon((i,j))= \pm 1$ and $\phi$ is a
bijection of the vertex set of $\mathbb{T}_n$. If $(i_1,j_1)$ is a
neighbor of $(i,j)$ then $\theta(X((i_1,j_1)))$ and $\theta(X((i,j)))$
share a signed edge vector. Hence $\phi((i_1,j_1))$ is a neighbor
of $\phi((i,j))$, i.e., $\phi$ is a  graph automorphism,
 and $\varepsilon((i_1,j_1))) = \varepsilon((i,j))$. By a connectedness argument,  $\varepsilon((i,j))$ is constant. If $\epsilon((i,j)) \equiv 1$ then $\theta=\hat\phi$
 and if $\varepsilon((i,j)) \equiv -1$ then $-\theta=\hat\phi$. \end{proof}

\begin{remark} \label{rem: Whitney} By a  classical result of Whitney \cite{Whi33} (see also
\cite[Section 5.3]{Oxl11}),  Lemma~\ref{lem: induced}  is valid for any $3$-connected graph.
 For graphs which are not $3$-connected there can
be a rich set of cycle-preserving bijections which do not
correspond to an automorphism of the graph. The cases of some of
such graphs (Diamond graphs, Laakso graphs, etc.) were studied in
\cite{DKO,DKO2}.
\end{remark}

\begin{lemma}  \label{lem: automorphism} Let $n \ge 5$. Then $H_n = U_n\cdot V_n = V_n \cdot U_n$. \end{lemma}
\begin{proof}  Let $g \in H_n$.
Since $n \ge 5$, $g$ maps the vertices $\{(i,j), (i+1,j), (i+1,j+1), (i+1,j)\}$ of a `small square'  cycle onto the vertices of another small square. Hence $g$ maps collinear triples
$\{(i,j),(i+1,j), (i+2,j)\}$  and $\{(i,j), (i,j+1), (i,j+2)\}$ onto other collinear triples with `midpoints' mapped to midpoints.  There exists $v \in V_n$ such that $vg((0,0)) = (0,0)$.  Since $vg$ maps collinear triples onto collinear triples, it follows that there exists $u \in U_n$ such that $uvg$ fixes $(0,0)$, $(1,0)$ and $(0,1)$. Thus, $uvg$ also fixes $(1,1)$.
Using the observation about  collinear triples again, it follows easily that $uvg = I$, i.e., $g \in U_n \cdot V_n$.

Similarly, we prove $g \in V_n \cdot U_n$.
\end{proof}

\begin{remark} \label{rem: n=4} Lemma~\ref{lem: automorphism} is false for $n=4$. This is because there are automorphisms which do not map collinear triples onto collinear triples: consider $\phi: \mathbb{T}_4 \rightarrow \mathbb{T}_4$  given by  \begin{align*}
&(2,0) \mapsto (1,3), (3,0) \mapsto (0,3), (2,1) \mapsto (1,2), (3,1) \mapsto  (0,2)\\
&(0,2) \mapsto (3,1), (1,2) \mapsto(2,1), (0,3) \mapsto(3,0), (1,3) \mapsto(2,0),  \end{align*}
with all other vertices fixed by $\phi$.  Then $\phi$ is easily seen to be a graph automorphism which maps the ``small square" cycle
$((0,1), (1,1), (1,2), (0,2),(0,1))$ onto the ``row"  cycle  $((0,1), (1,1), (2,1), (3,1), (0,1))$.
\end{remark} We require the following simple observation.
\begin{lemma} \label{lem: crossfixed} For  $g \in H_n$,  $\hat{g}(X((i,j))) = X(g((i,j)))$. \end{lemma}
\begin{lemma} \label{lem: invariantA} Let $n \ge 5$. Suppose $Y \in Z_n$ is $U_n$-invariant (i.e.,   $\hat{u}(Y)=Y$ for all $u \in U_n$). Then $Y$ induces a unique $\Gamma_n$-invariant mapping $A \colon B_n \rightarrow Z_n$ satisfying $A(X((0,0))) = Y$. Conversely, if $A$ is $\Gamma_n$-invariant then $Y = A(X((0,0)))$ satisfies $\hat{u}(Y) = Y$ for all $u \in U_n$.
\end{lemma}\begin{proof}  Let $$Y = \sum_{(i,j) \ne (0,0)} a(i,j) S((i,j)) + bR + cK$$
be the basis expansion of $Y$. Note that $\hat{\alpha}(K)= - K$, $\hat{\alpha}(R) = R$, and $\hat{\alpha}(S((i,j)) )= - S((-i-1,j))$. Hence
$$\hat{\alpha}(Y) =- \sum_{(i,j) \ne (0,0)} a(i,j) S((-i-1,j))  +bR - cK.$$
Since $Y = \hat{\alpha}(Y)$, it follows that $c=0$. Similarly, $b=0$. Hence\begin{equation}\label{eq: Y}  \begin{split}
 \sum_{(i,j)\ne (0,0)}\hat{T}_{(i,j)}(Y)
&=  \sum_{(k,l)\ne (0,0)} a(k,l) (\sum_{(i,j)\ne (0,0)}\hat{T}_{(i,j)}(S((k,l))))\\&=- \sum_{(k,l)\ne (0,0)} a(k,l)S((k,l)) = -Y. \end{split} \end{equation}

 For $(i,j) \ne (0,0)$, let $A(X((i,j))) = \hat{T}_{(i,j)}(Y)$ and extend linearly to define $A$ on  $B_n$. Note that by \eqref{eq: Y}
$$A(X((0,0))) = A(-\sum_{(i,j)\ne (0,0)} X((i,j))) = -\sum_{(i,j)\ne (0,0)}\hat{T}_{(i,j)}(Y) =  Y $$
as desired. Now let us check that  $A$ is $\Gamma_n$-invariant.  Let $\gamma \in  \Gamma_n$. Then   $\gamma = \pm \hat{g}$ for some $g \in H_n$.
So it suffices to check that $\hat{g}A = A\hat{g}$.  Fix $(i,j) \in \mathbb{T}_n$. Since $H_n = V_n \cdot U_n$, it follows that $gT_{(i,j)} = T_{g((i,j))} u$ for some $u \in U_n$.
Hence \begin{align*}\hat{g} A(X((i,j))) &= \hat{g} \hat{T}_{(i,j)}(Y)\\&= \hat{T}_{g((i,j))}\hat{u}(Y)\\ & =\hat{T}_{g((i,j))}(Y)\\&= A(X(g((i,j))))\\& = A\hat{g}(X((i,j))).
\end{align*}
Now suppose that $A \colon B_n \rightarrow Z_n$ is
$\Gamma_n$-invariant.  By Lemma~\ref{lem: crossfixed}, for all $u
\in U_n$,
$$\hat{u}(A(X(0,0)))= A(\hat{u}(X((0,0))))= A(X((0,0)))$$ and
$$A(X((i,j)))= A(\hat{T}_{(i,j)}(X((0,0))))= \hat{T}_{(i,j)}(A(X(0,0))).$$ So  $Y=A(X((0,0)))$ is $U_n$-invariant  and $A$
is uniquely determined by $A(X((0,0)))$.
\end{proof}

\begin{lemma} \label{lem: invariantvectors} Suppose $n \ge 3$.   Let $Y =  \sum_{(i,j) \ne (0,0)} a(i,j) S((i,j)) + bR + cK \in Z_n.$ Then $Y$ is $U_n$-invariant if and only if
\begin{enumerate}
\item $b=c=0$,
\item $a(i,i)=0$ and $a(i,-i-1)=0$ for $0 \le i \le n-1$,
\item if $n = 2k+1$ is odd then $a(i,k) = a(k,j) =0$ for all $0 \le i,j \le n-1$,
\item if $j \ne i$ and $j \ne -i-1$ (and, if $n=2k+1$, then  $i \ne k$ and $j \ne k$)  then
\begin{align*} a(i,j)&=-a(-i-1,j)=-a(j,i) = a(-j-1,i) = -a(i,-j-1)\\ &= a(-i-1,-j-1) = -a(-j-1,-i-1)= a(j,-i-1). \end{align*}
\end{enumerate}
\end{lemma}

\begin{proof}  It was proved above that if $Y$ is $U_n$-invariant then $b=c=0$. For the remaining conditions, consider
 the orbit of $S((i,j))$ under the action of $U_n$ as given by the following table:
\begin{center}
\begin{tabular}{c c}
$u \in U_n$ & $\hat{u}(S((i,j)))$ \\ \hline  $I$ & $S((i,j))$\\ $\alpha$ & $-S((-i-1,j))$\\
$\beta$ & $-S((j,i))$\\ $\alpha\beta$ & $S((-j-1,i))$\\ $\beta\alpha\beta$ & $-S((i,-j-1))$\\
$(\alpha\beta)^2$ & $S((-i-1,-j-1))$ \\ $\beta(\alpha\beta)^2$ & $-S((-j-1,-i-1))$ \\
$(\alpha\beta)^3$ &  $S((j,-i-1))$\\
\end{tabular}
\end{center} We determine necessary conditions for   $\hat{u}(Y)=Y$ to hold. Referring to the table, setting $u = \beta$  implies $a(i,i)=-a(i,i)$, i.e., $a(i,i)=0$ which is the first part of (2).  Setting  $u = \beta(\alpha\beta)^2$,  and taking into account that $$\hat{u}(S((n-1,n-1)) =
 - S((0,0)) = \sum_{(i,j) \ne (0,0)} S((i,j)),$$ yields
$$ a(i,-i-1)=  -a(i,-i-1) + a(n-1,n-1) = -a(i,-i-1)+0  \Rightarrow a(i,-i-1)=0,$$   which is the second part of  (2). Now suppose that $n = 2k+1$ is odd.
Setting $u = \alpha$, and taking into account that  $\hat{\alpha}(S((n-1,0))) = -S((0,0)) = \sum_{(i,j)\ne(0,0)}S((i,j))$ and that (from (2)) $a(n-1,0)=0$, yields
$$  a(k,j)  = -a(k,j)  + a(n-1,0) = -a(k,j)+ 0 =-a(k,j) \Rightarrow a(k,j)=0.$$ Setting $u = \beta$ gives $a(i,k) = -a(k,j)= 0$. So (3) is proved. Finally,  if  $(i,j)$ satisfies the assumptions of (4) then the orbit of $S((i,j))$ under $U_n$ has cardinality $8$ and (4) follows from the table by considering each $u \in U_n$ in turn. (Note
that  if $\hat{u}(S((i,j))) =\pm S((0,0))$ then $(i,j) \in \{(0,n-1),(n-1,0),(n-1,n-1)\}$ and so $a(i,j)=0$ from (1) and (2).)

This completes the proof of necessity of (1)-(4). Sufficiency is easily checked.
\end{proof}

\begin{corollary} \label{cor: dimensioninvariant} Let $n \ge 3$. Then the collection of $U_n$-invariant vectors $Y \in Z_n$ is a linear subspace of dimension $[n/2]([n/2]-1)/2$
with basis consisting of the vectors of the form
\begin{align*}V&= S((i,j)) -S((-i-1,j))-S((j,i))\\ &+S((-j-1,i))-S((i,-j-1))+ S((-i-1,-j-1))\\&-S((-j-1,-i-1))+ S((j,-i-1)),
 \end{align*} where $(i,j)$ satisfies the assumptions of (4) above.
\end{corollary} \begin{theorem} \label{thm: invariantprojections}  Let $n \ge 5$.
 The collection of all   $\Gamma_n$-invariant projections $P$ on $E_n$  with range $Z_n$ is  parameterized by the linear space of
$\Gamma_n$-invariant mappings $A \colon B_n \rightarrow Z_n$ which
has dimension $[n/2]([n/2]-1)/2$. \end{theorem} \begin{proof}
Note that
 $\ell_2(E_n) = Z_n\oplus B_n$. With respect to this  decomposition the  $\Gamma_n$-invariant projections are of the form
$$ P= \begin{bmatrix} I & A\\ 0 &0 \end{bmatrix},$$
where $A \colon B_n \rightarrow Z_n$ is $\Gamma_n$-invariant.  The
result follows from the previous lemma.
\end{proof} \begin{corollary} The orthogonal projection $P_{\orth}$ is the unique $\Gamma_n$-invariant projection if and only if $n \le 4$. \end{corollary}
\begin{proof} We omit  the easy proof  of  uniqueness in the  case $n=2$.  Note that $[n/2]([n/2]-1)/2>0$ if and only if   $n \ge 4$, so uniqueness for $n=3$ follows from Corollary~\ref{cor: dimensioninvariant} and non-uniqueness for $n \ge 5$ follows from
 Theorem~\ref{thm: invariantprojections}. So suppose $n = 4$.  By Lemma~\ref{lem: invariantvectors} the  linear space of $U_4$-invariant vectors is one-dimensional
with basis vector \begin{align*}
Y &= S((1,0)) - S((2,0)) + S((3,1)) - S((3,2))\\ &+ S((2,3))- S((1,3)) + S((0,2)) - S((0,1)). \end{align*}
Consider the automorphism $\phi \colon \mathbb{T}_4 \rightarrow \mathbb{T}_4$ described in Remark~\ref{rem: n=4}. Suppose, to derive a contradiction,  that
there exists a $\Gamma_4$-invariant projection $P$ with
$P((X(0,0))) = Y$. Note that  $\hat{\phi}(X((0,0))) = X((0,0))$, and so
$$\hat{\phi}(Y)= \hat{\phi}(P(X((0,0))) )= P(\hat{\phi}(X((0,0)))) =P(X((0,0))) = Y.$$
However, a direct calculation shows that
 \begin{align*}
\hat{\phi}(Y) &= S((1,0)) + 2S((2,0)) - S((0,1)) - 2S((3,1))\\
&- 2S((0,2)) - S((3,2)) + 2S((1,3)) + S((2,3)). \end{align*}
So $\hat{\phi}(Y) \ne Y$, which is the desired contradiction. This proves uniqueness in the case $n =4$.
\end{proof}
\subsubsection{Minimal projections} If $P$ is a
$\Gamma_n$-invariant projection onto $Z_n$, then $\|P\|_1 =
\|P(e)\|_1$ for any edge vector $e$ since the action of $\Gamma_n$
on $E_n$ is transitive. So Theorem~\ref{thm: invariantprojections}
can be used to compute the norm of an invariant  minimal
projection $P_{\min}$. The problem reduces to minimizing a convex
piecewise-linear function of $[n/2]([n/2]-1)/2$ variables. In the
case $n=5$, the family of $\Gamma_n$-invariant projections is
one-dimensional but
 $P_{\orth}$ is the unique minimal  $\Gamma_n$-invariant projection.

For $n = 6$, however,  $P_{\orth}$ is not a minimal projection.  The corresponding piecewise-linear
function of $3$ variables was minimized with the aid of MAPLE.  The results are as follows.

\begin{proposition} \label{prop: n=6} There is a unique $\Gamma_6$-invariant minimal projection $P_{\min}$ from $\ell_1((E_6))$ onto $Z_6$ with
$$\|P_{\min} \|_1 = \frac{109}{36}\quad\text{and}\quad \|I-P_{\min}\|_1 = 3.$$
On the other hand,
$$\|P_{\orth}\|_1 = \frac{3839}{1260}\quad\text{and}\quad \|I-P_{\orth}\|_1= \frac{317}{105}.$$
\end{proposition}  For $n \le 5$,  $P_{\orth}$ is the unique $\Gamma_n$-invariant minimal  projection and the relevant norms are as summarized in the following table:

\begin{center} \
\begin{tabular}{c c c } \label{tab: projnorms}
$n$ & $\|P_{\orth}\|_1$ & $\|I-P_{\orth}\|_1$ \\ \hline
$2$ &  $1$ & $3/2$ \\ $3$ & $19/9$ & $2$ \\
$4$ & $41/16$ & $5/2$ \\
$5$ & $69/25$ & $68/25$
\end{tabular}
\end{center} \begin{proposition} \label{prop: projconstant} Let $n \ge 3$ and let $P_{\min}$ be a minimal $\Gamma_n$-invariant projection from $\ell_1(E_n)$ onto $Z_n$.
Then $$\lambda(\lip_0(\mathbb{T}_n)) = \|I-P_{\min}\|_1 = \|P_{\min}\|_1 - \frac{1}{n^2}.$$
\end{proposition} \begin{proof}
Let $P$ be any $\Gamma_n$-invariant projection onto $Z_n$.
Note that $\langle P(e), e\rangle$ is independent of $e \in E_n$. So
$$ \langle P(e), e\rangle =\frac{\operatorname{trace}(P)}{|E_n|} = \frac{\operatorname{dim}(Z_n)}{2n^2}=\frac{n^2+1}{2n^2}. $$
From \eqref{eq: I-P}, we obtain
  \begin{equation} \label{eq: I-P} \|I - P\|_1 = \|e - P(e)\|_1= \|P(e)\|_1 + 1 -2 \langle P(e), e\rangle = \|P\|_1 - \frac{1}{n^2}. \end{equation}  Now suppose that $Q$ is a minimal projection from $\ell_\infty(E_n)$ onto $B_n$. We may  assume  that $Q$ is $\Gamma_n$-invariant. Then  $P=I - Q^*$ is a $\Gamma_n$-invariant projection onto $Z_n$ (see Corollary~\ref{cor: T^*} below for a proof of this in greater generality).  Thus,
$$\|Q\|_\infty = \|Q^*\|_1 = \|I-P\|_1=\|P\|_1 - \frac{1}{n^2} \ge  \|P_{\min}\|_1 - \frac{1}{n^2}= \|I-P_{\min}\|.$$
The reverse inequality $\|Q\|_\infty \le \|I-P_{\min}\|_1$ follows
from the minimality of  $Q$ and the fact that $I-P_{\min}^*$ is a
projection onto $B_n$. Hence $$\lambda(\lip_0(\mathbb{T}_n)) =
\|Q\|_\infty = \|I - P_{\min}\|_1 = \|P_{\min}\|_1 -
\frac{1}{n^2}.$$
\end{proof}
\begin{corollary} (a)   For $n \le 5$, $\lambda(\lip_0(\mathbb{T}_n)) = \|I-P_{\orth}\|_1$ and \newline
 $c_1((\mathcal{P}(\mathbb{T}_n),d_{W_1}))  \le \|I - P_{\orth}\|_1$ as given in the table above.

(b) $\lambda(\lip_0(\mathbb{T}_6))=3$ and  $c_1((\mathcal{P}(\mathbb{T}_6),d_{W_1}))  \le 3$.
\end{corollary}

\begin{proof}  The estimates for $c_1(\mathcal{P}(\mathbb{T}_n),d_{W_1}))$ in (a) and (b) follow from Proposition~\ref{prop: minimalnormprops}
and those for $\lip_0(\mathbb{T}_n)$ follow from Propositions~\ref{prop: n=6} and \ref{prop: projconstant} and the table above. \end{proof}

\begin{remark}  The estimates for $c_1((\mathcal{P}(\mathbb{T}_n),d_{W_1}))$ above improve somewhat the estimates for small $n$
given by John's theorem \cite{J}  asserting
that if $Y$ is an $n$-dimensional normed space then   $d_{\BM}(Y,\ell_2^n) \le \sqrt{n}$:
$$c_1((\mathcal{P}(\mathbb{T}_n),d_{W_1})) =c_1(\tc(\mathbb{T}_n)) \le d_{\BM}(\tc(\mathbb{T}_n), \ell_2^{n^2-1})\le \sqrt{n^2-1},$$
where the first equality is a result of \cite{NS}.
\end{remark}

\begin{corollary} For all $n \ge 3$, $d_{\BM}(\tc(\mathbb{T}_n), \ell_1^{n^2 -1}) \ge \| P_{\min}\|_1 - 1/n^2$.
\end{corollary}

\begin{proof} By the proof of Proposition~\ref{prop: minimalnormprops}(iv),
\[d_{\BM}(\tc(\mathbb{T}_n), \ell_1^{n^2 -1}) \ge
\lambda(\lip_0(\mathbb{T}_n)) = \| P_{\min}\|_1 -
\frac{1}{n^2}.\qedhere\]
\end{proof}

Unfortunately, we do not have sharp asymptotic estimates for $\|P_{\min}\|_1$ or even $\|P_{\orth}\|_1$.  The following summarizes what is known.
\begin{proposition} \label{prop: asymptotic} There exist positive constants $a_1$ and $a_2$ such that for $n \ge 2$,
$$ a_1 \sqrt{\log(n)} \le \|P_{\min}\|_1 \le a_2 \log(n).$$
\end{proposition} \begin{proof} The lower estimate follows from the estimate
$c_1((\mathcal{P}(\mathbb{T}_n),d_{W_1}))  \ge  a_1
\sqrt{\log(n)}$ due to Naor and Schechtman \cite{NS} combined with
Proposition~\ref{prop: minimalnormprops}(ii). The upper estimate
is valid for all finite  graphs $G=(V,E)$, with $|V|=n$,  equipped
with the graph distance. It is  proved using  stochastic  tree
embeddings
 \cite{S}.
\end{proof}
We close this section  with a central open question related to our results.
\begin{question} Find good   asymptotic estimates (improving on Proposition~\ref{prop: asymptotic})  for $\|P_{\min}\|_1$ and/or  $\|P_{\orth}\|_1$.
\end{question}

\subsection{$m$-dimensional tori} \label{sec: higherdimensions}

In this section we consider invariant projections on the cycle
spaces of   $m$-dimensional  tori $\mathbb{Z}_n^m$ for $m,n \ge
2$.  These graphs are $3$-connected, so by Remark~\ref{rem:
Whitney} the cycle-preserving isometries of the edge space are
induced by graph automorphisms.

The cases $n=2,3$ can be placed in a more general setting of
Hamming graphs considered in Section \ref{S:HamGr}. Here we start
with the case $n=4$.

\subsubsection{The case $n=4$} Next we prove uniqueness for
$\mathbb{Z}_4^m$. In the next lemma we use the notation introduced
in Section~\ref{sec: torus}.

\begin{lemma}\label{lem: squarecycle} For each $(i,j) \in \mathbb{T}_4$ there exists $\psi \in H_4$ such that
$$\psi((0,0))=(0,0)\quad\text{and}\quad \hat{\psi}(S((i,j))) = -S((i,j)).$$
\end{lemma}

\begin{proof}
 This follows from the  table in the proof of Lemma~\ref{lem: invariantvectors} and from the definition of $\phi$ in
Remark~\ref{rem: n=4}.  Since $\hat{\beta}(S((i,j))) = -S((j,i)))$, it suffices to check the claim  in the range  $0 \le i < j \le 3$:
$$\hat{\psi}(S((0,1))) = -S((0,1))\quad\text{for}\quad \psi = \phi^{-1}\alpha\phi;$$
$$\hat{\psi}(S((0,2))) = -S((0,2))\quad\text{for}\quad \psi = \beta\phi.$$
$$\hat{\psi}(S((0,3))) = -S((0,3))\quad\text{for}\quad \psi =\phi^{-1}\alpha\phi;$$
$$\hat{\psi}(S((1,2))) = -S((1,2))\quad\text{for}\quad \psi =\phi^{-1}\beta\alpha\beta\phi;$$
$$\hat{\psi}(S((1,3))) = -S((1,3))\quad\text{for}\quad \psi =\phi;$$
\[\hat{\psi}(S((2,3))) = -S((2,3))\quad\text{for}\quad \psi
=\phi^{-1}\alpha\phi.\qedhere\]
\end{proof}

\begin{lemma} \label{lem: extensionfrom2tom}  Suppose $\psi \in H_4$ is given by
$$\psi((i_1,i_2)) = (j_1,j_2) \qquad ((i_1,i_2)\in \mathbb{T}_4).$$
Then, for each $m \ge 2$,  $\psi$ generates a graph automorphism $\Psi$ of $\mathbb{Z}_4^m$ given by
$$\Psi((i_1,i_2,i_3,\dots,i_m)) = (j_1,j_2,i_3,\dots,i_m) \qquad  ((i_1,\dots,i_m) \in \mathbb{Z}_4^m).$$
Moreover, if $\psi((0,0))=(0,0)$, then  $\Psi((0,\dots,0))=(0,\dots,0)$.
\end{lemma}    \begin{proof} Immediate from the definitions.
\end{proof}

\begin{lemma} \label{lem: spanningset} Let $m \ge 2$. The cycle space $Z(\mathbb{Z}_n^m)$ is spanned by the collection of all  vectors $v_S$ corresponding to  $(i,j)$-parallel
small square cycles and all $i$-parallel  cycle vectors  $v_R$ corresponding to an $i$-parallel row cycle R with vertices:
 $$(x_1,\dots,x_{i-1},\{0,1,2,\dots,n-1\},x_{i+1},\dots,x_m),$$
where $x_1,\dots, x_m$ are arbitrary.
\end{lemma} \begin{proof}   As in Example~\ref{ex: cayley} we view the graph associated to
$\mathbb{Z}_n^m$ as the Cayley graph of a group with generators $g_i$ ($1 \le i \le m $) satisfying the following relations:
\begin{equation} \label{eq: smallsquares}  g_{i}+g_{j}=g_{j}+g_{i}\qquad (i \ne j), \end{equation} and
\begin{equation} \label{eq: rowcycle}  \underbrace{g_i+\dots+g_i}_{n~{\rm times}}=0\qquad (i=1,\dots,m).
\end{equation}
The cycles corresponding to \eqref{eq: smallsquares} are $(i,j)$-parallel small squares and those corresponding to \eqref{eq: rowcycle} are $i$-parallel row cycles. Hence the result follows from Proposition~\ref{P:GenCase}.
\end{proof} \begin{theorem} Let $m \ge 2$.  Then  $P_{\operatorname{orth}}$ is the unique invariant projection on $Z(\mathbb{Z}_4^m)$.
\end{theorem} \begin{proof} Let $\Gamma_0(\mathbb{Z}_n^m)$ denote the group of automorphisms of $\mathbb{Z}_n^m$ which fix $(0,\dots,0)$.  Let $S$ be a small square cycle. Without loss of generality, we may assume that $S$ is $(1,2)$-parallel.
By Lemmas~\ref{lem: squarecycle} and \ref{lem: extensionfrom2tom} there exists $\psi_S \in \Gamma_0(\mathbb{Z}_4^m)$ such that
 $\hat{\psi}_S(v_S) = - v_S$.  Similarly, if $R$ is a $1$-parallel  row cycle (without loss of generality)  then
$\hat{\psi}_R(v_R) = - v_R$, where $\psi_R \in \Gamma_0(\mathbb{Z}_4^m)$ is given by  $\psi_R(x_1,\dots,x_n) = (-x_1,\dots,x_n)$. Hence, as in the proof of Lemma~\ref{lem: v=0},

$$\sum_{\psi \in \Gamma_0(\mathbb{Z}_4^m)} \hat{\psi}(v_S) = \sum_{\psi \in \Gamma_0(\mathbb{Z}_4^m)} \hat{\psi}(v_R) = 0.$$
By Lemma~\ref{lem: spanningset} it follows that if $v \in Z(\mathbb{Z}_4^m)$ is  $\Gamma_0(\mathbb{Z}_4^m)$-invariant, then $v = 0$.

 Let $X$  be the sum of all oriented edge vectors with initial vertex $(0,0,\dots,0)$. Note that $X$ is a $\Gamma_0(\mathbb{Z}_4^m)$-invariant
vector belonging to the cut space  $B(\mathbb{Z}_4^m)$.  Let $P$ be an invariant projection onto $ Z(\mathbb{Z}_4^m)$.
Then, for all $\phi \in  \Gamma_0(\mathbb{Z}_4^m)$,
$$ \hat{\phi}(P(X)) = P(\hat{\phi}(X)) = P(X).$$
Hence $P(X) \in Z(\mathbb{Z}_4^m)$ is $\Gamma_0(n,m)$-invariant,
 so $P(X)=0$ from the  above.
The argument  given in Lemma~\ref{lem: invariantA} for $m=2$
readily extends to $m \ge 2$ to show that $P$ annihilates
$B(\mathbb{Z}_4^m)$, i.e., $P = P_{\rm{orth}}$.
\end{proof}

\subsubsection{The case $n\ge5$} Next we show that for $n \ge 5$
and $m \ge2 $ there exist nonorthogonal invariant projections onto
$Z(\mathbb{Z}_n^m)$.
\begin{proposition} \label{prop: isometryZn} Let  $n \ge 5$ and $m \ge 2$.  For each $\psi \in \Gamma_0(\mathbb{Z}_n^m)$ there exist a permutation $\sigma$ of $\{1,\dots,m\}$ and signs $\alpha_j = \pm1$ ($1 \le j \le m$) such that \begin{equation} \label{eq: psi}
\psi((i_1,\dots,i_m)) = (\alpha_1 i_{\sigma(1)},\dots,\alpha_m i_{\sigma(m)}). \end{equation}
\end{proposition} \begin{proof} Since $\psi$ fixes $(0,\dots,0)$ it is clear that there is a unique choice of  $\sigma$ and $\alpha_j = \pm 1$ such that
\eqref{eq: psi} holds for all points of the form $(0,\dots,0,1,0,\dots,0)$.

Without loss of generality, we may assume that $\sigma$ is the identity permutation and $\alpha_j = 1$ ($1 \le j \le m$).  Hence $\psi$ fixes each point of the form $(0,\dots,0,1,0,\dots,0)$. By induction on $\operatorname{card} \{1 \le j \le m \colon i_j = 1\}$, it follows that  \begin{equation}\label{eq: r=0} \text{$\psi(x) = x$
for $x = (i_j)_{j=1}^m$, where $i_j \in \{0,1\}$} \end{equation}
We shall prove by induction that  $\psi(x) = x$ for all $x  \in \mathbb{Z}_n^m$. To that end, let $0 \le r \le m$ and suppose that \begin{equation}
\label{eq: inductiononr}
 \text{
$\psi(x) = x$  if $i_j \in \{0,1\}$ for all $j>r$ and  $i_j \in \mathbb{Z}_n$ for all $j \le r$}. \end{equation}
Note that, by \eqref{eq: r=0}, $r=0$ satisfies \eqref{eq: inductiononr}. Suppose that \eqref{eq: inductiononr} is satisfied for some $r$, where $0 \le r < m$.
Fix values  $x_j \in \mathbb{Z}_n$ for $1 \le j \le r$ and  $x_j \in \{0,1\}$ for $j \ge r+2$. Since \eqref{eq: inductiononr} is satisfied, it follows that
$$\psi((x_1,\dots,x_r, i_{r+1},x_{r+2},\dots x_m) )= (x_1,\dots,x_r, i_{r+1},x_{r+2},\dots, x_m)$$
provided $i_{r+1} \in \{0,1\}$. Since $n \ge 5$, it follows  as in the proof of Lemma~\ref{lem: automorphism}  that $\psi$ maps `collinear triples' onto collinear triples. Hence  \eqref{eq: inductiononr} is satisfied by $r+1$, which completes the proof of the inductive hypothesis. So \eqref{eq: inductiononr} is satisfied by $r=m$, i.e., $\psi$  is the identity mapping as claimed.
\end{proof} \begin{theorem} Let $n \ge 5$ and $m \ge 2$.  There exist nontrivial invariant projections onto $Z(\mathbb{Z}_n^m)$. \end{theorem}
\begin{proof} Let $S$ be the small square cycle oriented as  \begin{align*}
&(0,1,0,\dots,0) \rightarrow (1,1,0,\dots,0) \rightarrow (1,2,0,\dots,0) \rightarrow\\
&(0,2,0,\dots,0) \rightarrow (0,1,0,\dots,0)
\end{align*} and let $v_S \in Z(\mathbb{Z}_n^m)$ be the corresponding cycle vector. Let $$w = \sum_{\psi \in \Gamma_0(\mathbb{Z}_n^m)} \hat{\psi}(v_S).$$Using Proposition~\ref{prop: isometryZn},  the coefficient of $w$ on the edge vector $(1,1,0,\dots,0) \rightarrow (1,2,0,\dots,0)$ equals
$$\operatorname{card} \{\psi \in \Gamma_0(\mathbb{Z}_n^m) \colon \sigma(1)=1, \sigma(2) = 2, \alpha_1 = \alpha_2 =1\} =2^{m-2} (m-2)! \ne 0.$$
Hence $w \in Z(\mathbb{Z}_n^m)$ is a nontrivial $\Gamma_0(\mathbb{Z}_n^m)$-invariant vector.

 Let $X$  be the sum of all oriented edge vectors with initial vertex $(0,0,\dots,0)$. Note that $X$ is a $\Gamma_0(\mathbb{Z}_n^m)$-invariant
vector belonging to the cut space of $\mathbb{Z}_n^m$. The argument  given in Lemma~\ref{lem: invariantA} for $m=2$ readily extends to $m \ge 2$ to show that there
is a unique invariant projection $P$ onto $Z(\mathbb{Z}_n^m)$ satisfying $P(X) = w$.
\end{proof} \begin{remark} For $m >2 $, we do not  have a simple description of  the invariant projections similar to
Theorem~\ref{thm: invariantprojections}. This is related to the fact that we do not have a simple algebraic basis for $Z(\mathbb{Z}_n^m)$
analogous to the basis  of small square cycle vectors of Lemma~\ref{lem: bases} for $m=2$.
\end{remark}

\subsection{Hamming graphs}\label{S:HamGr}

 For $n
\ge 2$, let  $A_n = \{0,\dots,n-1\}$ be an `alphabet' of size $n$.
Equip $A_n^m$ with the Hamming metric: $$d((a_i)_{i=1}^m,
(b_i)_{i=1}^m) = \operatorname{card} \{ i \colon a_i \ne b_i\}.$$
We call  $A_n^m$  the \textit{Hamming graph} because we regard it
as the vertex set of a graph with edges
$\{(a_i)_{i=1}^m,(b_i)_{i=1}^m\}$ such that $d((a_i)_{i=1}^m,
(b_i)_{i=1}^m) =1$. Note that $A_n^m$ and $\mathbb{Z}_n^m$ are
isometric if and only if $n \le 3$. Hamming graphs is a well-known
object in algebraic combinatorics (see \cite[Section 9.2]{BCN89}).

Let $\Gamma(n,m)$ denote the group of graph automorphisms of
$A_n^m$ and let $\Gamma_0(n,m)$ denote its subgroup consisting of
the automorphisms which fix $(0,0,\dots,0)$.

\begin{lemma} Each $\phi \in \Gamma_0(n,m)$ is uniquely determined by a permutation $\sigma$ of $\{1,2,\dots,m\}$ and permutations $\phi_i$ ($1 \le i \le m$)
of $A_n$  such that  $\phi_i(0)=0$ and such that \begin{equation}
\label{eq: automorphofGamma(N,m)} \phi((a_i)_{i=1}^m) =
(b_i)_{i=1}^m, \quad \text{where}\quad  b_{\sigma(i)} =
\phi_i(a_i). \end{equation}
\end{lemma} \begin{proof} Since $\phi$ fixes $(0,0,\dots,0)$,  it follows that $\sigma$ and $\phi_i$ exist such that  \eqref{eq: automorphofGamma(N,m)}
holds provided $\operatorname{card}\{i \colon a_i \ne 0\} \le 1$.
Hence  \eqref{eq: automorphofGamma(N,m)} holds for all
$(a_i)_{i=1}^m$ by a straightforward induction on
$\operatorname{card}\{i \colon a_i \ne 0\}$.
\end{proof}
A  `small square' cycle is said to be $(i,j)$-parallel, where $1
\le i < j \le m$ if  the   vertices of the corresponding  square
$S$ are of the form:
$$(x_1,\dots,x_{i-1},\{a_1,a_2\},x_{i+1}\dots,x_{j-1},\{b_1,b_2\},x_{j+1},\dots,x_m).
$$  Let $v_S$ denote the corresponding cycle vector belonging to the cycle space  $Z(A_n^m))$.
A `small triangle'  cycle is said to be $i$-parallel if the
vertices of the corresponding  triangle $T$ are of the form
$$ (x_1,\dots,x_{i-1},\{a_1,a_2,a_3\},x_{i+1},\dots, x_m).$$
Let $v_T \in  Z(A_n^m)$ denote the corresponding cycle vector.

\begin{lemma} \label{lem: squaresandtrianglesspan}$Z(A_n^m)$ is spanned by the collection of signed indicator functions of all small square
and all small triangle cycles. \end{lemma}

\begin{proof} Clearly, it suffices to prove the following
statement: Let $f\in Z(A_n^m)$ be the signed indicator function of
a directed cycle $C$ of length $k$. Then, adding to $f$ some
signed indicator functions of triangular and square cycles
described above, we get a sum of signed indicator functions of
directed cycles of length $\le k-1$ each.

We need to consider the following cases:

\begin{enumerate}[{\bf (1)}]

\item The cycle $C$ contains two consecutive edges for which the
ends differ in the same coordinate. Let the edges be
$\overrightarrow{uv}$ and $\overrightarrow{vw}$. In this case
$\overrightarrow{vu}, \overrightarrow{uw}$, and
$\overrightarrow{wv}$ form a triangle of the described type, and
its ``addition'' decreases the ``length'' of $C$.

\item Any two consecutive edges change different coordinates. Note
that each coordinate which ever changes in the cycle $C$, has to
be changed at least twice. Let $l$ be the smallest number of edges
between any two edges changing the same coordinate. Let
$\overrightarrow{uv}$ and $\overrightarrow{xy}$ be these two edges
changing the same coordinate, and $\overrightarrow{vw}$ be an edge
in $C$ on the shortest path from $v$ to $x$.

We prove the following statement (as is easy to see, it achieves
the goal described above). There is a square cycle of the
described type such that adding its signed indicator function to
$f$ we get one of the following four outcomes:

\begin{enumerate}[{\bf (a)}]

\item A signed indicator function of a cycle of length $k$ in
which the smallest number of edges between two edges changing the
same coordinate is equal to $l-1$,

\item Empty cycle,

\item A signed indicator function of a cycle of length $k-2$,

\item A union of two nontrivial cycles of total length $k$.

\end{enumerate}

In fact, without loss of generality, assume that

\[u=(x_1,x_2,x_3,\dots, x_m),\]
\[v=(\tilde x_1,x_2,x_3,\dots, x_m),\]
\[w=(\tilde x_1,\tilde x_2,x_3,\dots, x_m).\]

We introduce $\bar v=(x_1,\tilde x_2,x_3,\dots, x_m)$. There are
four possibilities:

\begin{enumerate}[{\bf (i)}]

\item $\bar v$ is not a vertex of the cycle $C$.

In this case, we add to $f$ the signed indicator function of the
cycle $w, v, u, \bar v, w$, and get a cycle of length $k$ in which
the  edge $\overrightarrow{\bar vw}$ and $\overrightarrow{xy}$
change the same coordinate, and the number of edges between them
is $l-1$.

\item $\bar v$ is a vertex of the cycle $C$, and the cycle has
four vertices  $w, v, u, \bar v, w$.

In this case, adding to $f$ the signed indicator function of the
cycle $w, v, u, \bar v, w$, we get $0$.

\item $\bar v$ is a vertex of the cycle $C$, and exactly one of
the edges $u\bar v$ and $\bar vw$ is an edge of $C$.

In this case, adding to $f$ the signed indicator function of the
cycle $w, v, u, \bar v, w$, we get the signed indicator function
of a cycle with $k-2$ edges.

\item $\bar v$ is a vertex of the cycle $C$, and none of the edges
$u\bar v$ and $\bar vw$ is an edge of $C$.

In this case, adding to $f$ the signed indicator function of the
cycle $w, v, u, \bar v, w$, we get the sum of signed indicator
function of two nontrivial cycles with total length of $k$ edges.

\end{enumerate}

\end{enumerate}

\end{proof}

\begin{lemma} \label{lem: v=0} Suppose $v \in  Z(A_n^m)$ is $\Gamma_0(n,m)$-invariant (i.e.,   $\hat{\phi}(v) = v$ for all $\phi \in \Gamma_0(n,m)$).
Then $v=0$. \end{lemma}

\begin{proof} Suppose $v$ is $\Gamma_0(n,m)$-invariant. Then $$\frac{1}{\operatorname{card}(\Gamma_0(n,m))} \sum_{\phi \in \Gamma_0(n,m)}
\hat{\phi}(v) = v.$$ So by Lemma \ref{lem:
squaresandtrianglesspan} it suffices to check that
$$ \sum_{\phi \in \Gamma_0(n,m)}\hat{\phi}(v_S)= \sum_{\phi \in \Gamma_0(n,m)}\hat{\phi}(v_T)=0$$
for all small squares $S$ and small triangles $T$.

Let $S$ be a small square. Without loss of generality, we may
assume that $S$ is (1,2)-parallel.

Case I:  $S$ has  edges oriented as follows: \begin{align*}
&(0,0,x_3,\dots,x_n) \rightarrow (a,0,x_3,\dots,x_n) \rightarrow
(a,b,x_3,\dots,x_n)\rightarrow\\&(0,b,x_3,\dots,x_n)\rightarrow
(0,0,x_3,\dots,x_n). \end{align*} where $a \ne 0, b\ne0$ and
$x_3,\dots,x_n$ are arbitrary.  Consider $\psi_S \in
\Gamma_0(n,m)$ corresponding to permutations  $\sigma$ and
$\psi_i$ ($1 \le i \le m$), where $\sigma$ is the transposition
which interchanges $1$ and $2$, and $\psi_1 = \psi_2$ is the
transposition which interchanges $a$ and $b$ if $a \ne b$ or  the
identity if $a = b$, and $\psi_i$ is the identity for $3 \le i \le
m$. Then $\psi_S$ maps $S$ onto the same square but with the
opposite orientation:
\begin{align*}
&(0,0,x_3,\dots,x_n) \rightarrow (0,b,x_3,\dots,x_n) \rightarrow
(a,b,x_3,\dots,x_n)\rightarrow\\&(a,0,x_3,\dots,x_n)\rightarrow
(0,0,x_3,\dots,x_n), \end{align*} So $\hat{\psi_S}(v_S) = -v_S$.
Hence, using the group property of $\Gamma_0(n,m)$,
\begin{align*} \sum_{\phi \in \Gamma_0(n,m)}
\hat{\phi}(v_S)&= \frac{1}{2}\sum_{\phi \in
\Gamma_0(n,m)}(\hat{\phi}(v_S) +
\hat{\phi}\hat{\psi_S}(v_S))\\&=\frac{1}{2}\sum_{\phi \in
\Gamma_0(N,m)} \hat{\phi}(v_S+ \hat{\psi_S}(v_S) )=0.\end{align*}

Case II: $S$ has edges oriented as follows:
 \begin{align*}
&(a_1,b_1,x_3,\dots,x_n) \rightarrow (a_2,b_1,x_3,\dots,x_n)
\rightarrow
(a_2,b_2,x_3,\dots,x_n)\rightarrow\\&(a_1,b_2,x_3,\dots,x_n)\rightarrow
(a_1,b_1,x_3,\dots,x_n), \end{align*} where $a_1 \ne 0$, $a_2 \ne
0$, and $x_3,\dots,x_n$ are arbitrary. Consider $\psi_S \in
\Gamma(n,m)$, where $\psi_1$ is the transposition $(a_1,a_2)$ and
$\psi_i$ ($2 \le i \le n$) and $\sigma$ are the identity. Then
$\hat{\psi_S}(v_S) = - v_S$ and hence $\sum_{\phi \in
\Gamma_0(n,m)}\hat{\phi}(v_S)=0$.

Let $T$ be a small triangle. Without loss of generality, $T$ is
$1$-parallel with oriented edges as follows:
$$(a_1,x_2,\dots,x_n) \rightarrow (a_2,x_2,\dots,x_n) \rightarrow (a_3,x_2,\dots,x_n)\rightarrow (a_1,x_2,\dots,x_n) ,$$
where $a_1 \ne 0$, $a_2 \ne 0$, and  $x_2,\dots,x_n$ are
arbitrary. Consider $\psi_T \in \Gamma(n,m)$, where $\psi_1$ is
the transposition $(a_1,a_2)$ and $\psi_i$ ($2 \le i \le n$) and
$\sigma$ are the identity. Then $\hat{\psi_T}(v_S) = - v_S$ and
hence $\sum_{\phi \in \Gamma(n,m)}\hat{\phi}(v_T)=0$.
\end{proof}

\begin{theorem} \label{thm: uniqueorth} The orthogonal projection $P_{\operatorname{orth}}$ is the unique $\Gamma(n,m)$-invariant projection onto $Z(A_n^m)$.
\end{theorem}

\begin{proof} Let $X$  be the sum of all oriented edge vectors with initial vertex $(0,0,\dots,0)$. Note that $X$ is a $\Gamma_0(n,m)$-invariant
vector belonging to the cut space of $A_n^m$.

Let $P$ be a $\Gamma(n,m)$-invariant projection onto  $Z(A_n^m)$.
Then, for all $\phi \in \Gamma_0(n,m)$,
$$ \hat{\phi}(P(X)) = P(\hat{\phi}(X)) = P(X).$$
So $P(X) \in Z(A_n^m)$ is $\Gamma_0(n,m)$-invariant. By
Lemma~\ref{lem: v=0}, $P(X)=0$. Let $a=(a_i)_{i=1}^m \in A_n^m$.
Clearly, $\Gamma(n,m)$ is transitive, so there exists $\psi_a \in
\Gamma(n,m)$ such that $\psi_a((0,0,\dots,0))= a$. Hence, by
$\Gamma(n,m)$-invariance,
$$P(\hat{\psi}_a(X))= \hat{\psi}_a(P(X))= \hat{\psi}_a(0) = 0.$$
Since $\{\hat{\psi}_a(X) \colon a \in A_n^m\}$ spans the cut
space, it follows that $P$ annihilates the cut space, i.e., $P =
P_{\operatorname{orth}}$.
\end{proof}

\begin{corollary}  \label{cor: 23} Let $m \ge 2$.  Then  $P_{\operatorname{orth}}$ is the unique invariant projection onto $Z(\mathbb{Z}_n^m)$ for $n=2$ and $n=3$.
\end{corollary}

\subsection{The Hamming Cube}

The main goal of this section is the following result.
\begin{theorem} \label{thm: normoforthoghammingcube} Let $Q_n$ be the orthogonal projection from the edge space of the Hamming cube $\mathbb{Z}_2^n$ onto the cut space. Then $\|Q_n\|_1 = \dfrac{n+1}{2}$.
\end{theorem}

\subsubsection{Notation}

First, let us settle on some notation for $\mathbb{Z}_2^n$  to
simplify the discussion: \begin{enumerate} \item Let $V_n =
\{(\varepsilon_i)_{i=1}^n \colon \varepsilon_i \in\{0,1\}\}$
denote the vertex set and let $E_n$ denote the edge set. The group
of graph automorphisms is denoted $\Gamma_n$. (Note that the
notation introduced here differs from that of Section~\ref{sec:
torus}.) \item  A directed edge $e$ with initial vertex $v = e_-$
and terminal vertex $w = e_+$ will be denoted as
$\overrightarrow{vw}$ or $\overrightarrow{e_-e_+}$. We will also
use the same notation  $\overrightarrow{vw}$ to  denote the
corresponding edge vector belonging to $E(\mathbb{Z}_2^n)$. \item
The edge, cycle, and cut spaces are denoted $E(\mathbb{Z}_2^n)$,
$Z(\mathbb{Z}_2^n)$, and $B(\mathbb{Z}_2^n)$. \item We say that an
oriented edge $e$ is $j$-parallel if $\varepsilon_j =0$  at $e_-$
and $\varepsilon_j =1$  at $e_+$. \item Let $e_0$ be the edge with
initial vertex $ (0,0,\dots,0)$ and terminal vertex  $
(1,0,\dots,0)$. \item For $0 \le k \le n-1$, let $v_k$ denote any
vertex $(\varepsilon_i)_{i=1}^n$ such that
$$ \varepsilon_1 = 0\quad\text{and}\quad \operatorname{card} \{i \colon \varepsilon_i = 1\} = k.$$
For $1 \le k \le n$, let $w_k$ be any vertex  $(\varepsilon_i)_{i=1}^n$ such that
$$ \varepsilon_1 = 1\quad\text{and}\quad \operatorname{card} \{i \colon \varepsilon_i = 1\} = k.$$
\item There are three types of edges:  $$\overrightarrow{w_k w_{k+1}} \quad (1 \le k \le n-1);$$
  $$\overrightarrow{v_k w_{k+1}} \quad(0 \le k \le n-1);$$
$$\overrightarrow{v_k v_{k+1}} \quad (0 \le k \le n-2).$$
Note that $e_0$ is the unique edge of the form
$\overrightarrow{v_0w_1}$. We use arrows to indicate our choice of
the reference orientation of $\mathbb{Z}_2^n$.

We will investigate the basis expansion of $Q_n(e_0)$ with respect to the standard edge basis $E_n$ of $E(\mathbb{Z}_2^n)$. To that end, let us define the  basis coefficients corresponding to the three edge types:
$$ a_k = \langle Q_n(e_0), \overrightarrow{w_k w_{k+1}}\rangle \quad(1 \le k \le n-1),$$
$$ b_k = \langle Q_n(e_0), \overrightarrow{v_k w_{k+1}}\rangle \quad(0 \le k \le n-1),$$
$$c_k = \langle Q_n(e_0), \overrightarrow{v_k v_{k+1}}\rangle \quad(0 \le k \le n-2),$$
Note that, by symmetry, $a_k$ does not depend on the choice of
edge of type  $ \overrightarrow{w_k w_{k+1}}$, and similarly for
$b_k$ and $c_k$. More precisely, we use the fact that for any two
vectors of the same type there is an automorphism of
$\mathbb{Z}_2^n$ mapping one of them onto the other and fixing
$e_0$.

\item For each $v \in V_n$, let $X(v) \in B(\mathbb{Z}_2^n)$
denote the cut vector which is the sum of the $n$  edge vectors
emanating from $v$  oriented so that $v$ is the initial vertex.
\end{enumerate}

\subsubsection{Technical lemmas about the basis coefficients of
$Q_n(e_0)$} Throughout this section we assume that $n \ge 3$.
\begin{lemma} \label{lem: a_kc_k} $a_k  = -c_{k-1}$ for $1 \le k \le n-1$.
\end{lemma} \begin{proof} Consider the translation of $\mathbb{Z}_2^n$ given by  $\phi(v) = v + (1,0,\dots,0)$  ($v  \in \mathbb{Z}_2^n$).
Then $Q_n \hat{\phi} = \hat{\phi}Q_n$ by $\Gamma_n$-invariance of $Q_n$. Hence \begin{align*}
a_k &= \langle Q_n(e_0), \overrightarrow{w_k w_{k+1}} \rangle\\&=  \langle \hat{\phi}(Q_n(e_0)),\hat{\phi}( \overrightarrow{w_k w_{k+1}}) \rangle\\
&=  \langle Q_n(\hat{\phi}(e_0)), \overrightarrow{v_{k-1} v_{k}} \rangle\\&= \langle Q_n(-e_0), \overrightarrow{v_{k-1} v_{k}} \rangle\\
&= -c_{k-1}.
\end{align*}
\end{proof} \begin{lemma} \label{lem: b_0}  $b_0 = \langle Q_n(e_0), e_0 \rangle = \dfrac{2}{n} - \dfrac{1}{n 2^{n-1}}.$
\end{lemma}  \begin{proof} By symmetry,  $\langle Q_n(f), f \rangle$ is independent of the choice of edge vector $f$. Hence \begin{align*}
\langle Q_n(e_0), e_0 \rangle &=  \frac{1}{\operatorname{card}(E_n)} \sum_{f \in E_n} \langle Q_n(f), f \rangle\\
&= \frac{\operatorname{trace}(Q_n)}{\operatorname{card}(E_n)}\\
&= \frac{\operatorname{dim}(B_n)}{\operatorname{card}(E_n)}\\
&= \frac{2^n-1}{n 2^{n-1}}\\ &= \frac{2}{n} - \frac{1}{n 2^{n-1}}.
\end{align*}
\end{proof} \begin{lemma} \label{lem: c_0} $c_0 = \dfrac{1}{n-1} - \dfrac{2}{n(n-1)} + \dfrac{1}{2^{n-1}n(n-1)}$.
\end{lemma} \begin{proof} \begin{align*} 1 &= \langle X((0,0,\dots,0)), e_0)\rangle \\
&= \langle X((0,0,\dots,0)),  Q_n(e_0)\rangle \\ &= \langle e_0,  Q_n(e_0)\rangle  + (n-1) \langle \overrightarrow{v_0v_1}, Q_n(e_0) \rangle\\
&= b_0 + (n-1)c_0.
\end{align*} Hence, by Lemma~\ref{lem: b_0},
$$ c_0 = \frac{1-b_0}{n-1}=  \dfrac{1}{n-1} - \dfrac{2}{n(n-1)} + \dfrac{1}{2^{n-1}n(n-1)}.$$
\end{proof}

\begin{lemma} \label{lem: a_k}   $a_k = \dfrac{b_k - b_{k-1}}{2}$  for $1 \le k \le n-1$.
\end{lemma}

\begin{proof} Consider a small square cycle with directed edges of the form
$$ v_{k-1} \rightarrow w_k \rightarrow w_{k+1} \rightarrow v_k \rightarrow v_{k-1}.$$
The corresponding cycle vector is  given by
$$\overrightarrow{v_{k-1}w_k}+\overrightarrow{w_{k}w_{k+1}}-\overrightarrow{v_{k}w_{k+1}}-\overrightarrow{v_{k-1}v_k}\in Z(\mathbb{Z}_2^n).$$
Since $Q_n(e_0) \in B(\mathbb{Z}_2^n)= Z(\mathbb{Z}_2^n)^\perp$,
we have
\begin{align*}0& = \langle Q_n(e_0), \overrightarrow{v_{k-1}w_k}+\overrightarrow{w_{k}w_{k+1}}-\overrightarrow{v_{k}w_{k+1}}-
\overrightarrow{v_{k-1}v_k}\rangle\\&= b_{k-1}+ a_k - b_k - c_{k-1}\\
&=  b_{k-1}+ a_k - b_k +a_{k}
\end{align*} by Lemma~\ref{lem: a_kc_k}. Hence the result follows.
\end{proof} \begin{lemma} \label{lem: b_1} $b_1 = \dfrac{2}{n(n-1)}-\dfrac{1}{2^{n-2}n(n-1)} - \dfrac{1}{n 2^{n-1}}$.
\end{lemma} \begin{proof} \begin{align*} b_1 &= 2a_1 + b_0\\ \intertext{(by  Lemma~\ref{lem: a_k})}
&= -2c_0+b_0  \\ \intertext{(by  Lemma~\ref{lem: a_kc_k})} &= -2[  \dfrac{1}{n-1} - \dfrac{2}{n(n-1)} + \dfrac{1}{2^{n-1}n(n-1)}]
+  \dfrac{2}{n} - \dfrac{1}{n 2^{n-1}}\\ \intertext{(by Lemmas~\ref{lem: c_0} and ~\ref{lem: b_0})}
&= \dfrac{2}{n(n-1)}-\dfrac{1}{2^{n-2}n(n-1)} - \dfrac{1}{n 2^{n-1}}.
\end{align*}  \end{proof}
\begin{lemma} \label{lem: a_krec}  For $2 \le k \le n-1$,
$$(n-k)a_k - (k-1)a_{k-1}-b_{k-1}=0.$$
\end{lemma}  \begin{proof} For $2 \le k \le n-1$, consider a cut vector of the form $X(w_k)$ for a particular choice of $w_k$.
There are $n-k$ edges of form $\overrightarrow{w_kw_{k+1}}$ with initial vertex $w_k$,  and $k-1$ edges of form $\overrightarrow{w_{k-1}w_k}$, and $1$ of form $\overrightarrow{v_{k-1}w_k}$, with terminal vertex $w_k$. Moreover, the support of  $X(w_k)$ is disjoint form $e_0$. Hence \begin{align*}
0 &= \langle e_0, X(w_k) \rangle \\
&=  \langle Q_n(e_0), X(w_k) \rangle \\
&= (n-k) \langle Q_n(e_0), \overrightarrow{w_kw_{k+1}}\rangle - (k-1) \langle Q_n(e_0), \overrightarrow{w_{k-1}w_{k}}\rangle  - \langle Q_n(e_0), \overrightarrow{v_{k-1}w_{k}}\rangle\\
&= (n-k)a_k - (k-1)a_{k-1} - b_{k-1}.
\end{align*}
 \end{proof}

 \begin{lemma} \label{lem: recurrence}
 For $2 \le k \le n-1$,
$$ (n-k)b_k - (n+1) b_{k-1} + (k-1)b_{k-2}=0.$$ \end{lemma}

\begin{proof} \begin{align*}
0 &= (n-k)a_k - (k-1)a_{k-1} - b_{k-1} \\ \intertext{(by
Lemma~\ref{lem: a_krec})} &= \frac{(n-k)}{2}(b_k - b_{k-1}) -
\frac{(k-1)}{2}(b_{k-1}-b_{k-2}) - b_{k-1}\\ \intertext{(by
Lemma~\ref{lem: a_k})} &=  \frac{1}{2} [ (n-k)b_k - (n+1) b_{k-1}
+ (k-1)b_{k-2}].\qedhere
\end{align*}
\end{proof}

\begin{lemma} \label{lem: b_{n-2}} $b_{n-2} = \dfrac{(n+1)}{n-1} b_{n-1}.$ \end{lemma}
\begin{proof} Note that $(0,1,1,\dots,1)$ is the unique vertex of form $v_{n-1}$. Since $e_0$ is not contained in the support of $X(v_{n-1})$,
 \begin{align*}
0& = \langle e_0, X(v_{n-1})\rangle \\
&= \langle Q_n( e_0), X(v_{n-1})\rangle \\
&=  \langle Q_n( e_0), \overrightarrow{v_{n-1}w_n} \rangle - (n-1)  \langle Q_n( e_0), \overrightarrow{v_{n-2}v_{n-1}} \rangle \\
&= b_{n-1} - (n-1)c_{n-2}\\
&= b_{n-1} + (n-1)a_{n-1}\\ \intertext{(by Lemma~\ref{lem: a_kc_k})}
&= b_{n-1} + \frac{(n-1)}{2} (b_{n-1}-b_{n-2})\\
\intertext{(by Lemma~\ref{lem: a_k})}
&= \frac{(n+1)}{2} b_{n-1} - \frac{(n-1)}{2} b_{n-2}.
\end{align*} Hence the result follows.
\end{proof} \begin{lemma} \label{lem: pos+dec} The sequence $(b_k)_{k=0}^{n-1}$ is  positive and strictly  decreasing.
\end{lemma} \begin{proof} First we prove positivity.  By Lemmas~\ref{lem: b_0} and \ref{lem: b_1}, $$b_0 > b_1 >0.$$
Suppose, to derive a contradiction, that $b_k \le 0$ and $b_k < b_{k-1}$ for some $k$ satisfying  $2 \le k \le n-1$. Note that, in fact, $k \le n-2$ since by Lemma~\ref{lem: b_{n-2}} $$b_{n-2} = \dfrac{(n+1)}{n-1} b_{n-1},$$
which cannot hold if  $b_{n-1} \le 0$ and $b_{n-1}< b_{n-2}$.

By Lemma~\ref{lem: recurrence}, with $k$ replaced by $k+1$ (as we may since $k \le n-2$), \begin{equation} \label{eq: convexcombo}
 \frac{(n+1)}{n-1} b_k = \frac{(n-k-1)}{n-1} b_{k+1} + \frac{k}{n-1} b_{k-1}. \end{equation}
Note that  the right-hand side  is a proper convex combination of $b_{k-1}$ and $b_{k+1}$.
However, by assumption, $b_k \le 0$ and $b_k < b_{k-1}$, so
$$ \frac{(n+1)}{n-1} b_k \le b_k < b_{k-1}.$$ Hence, from the strict convexity of  \eqref{eq: convexcombo} and our assumption that $b_k \le 0$,
$$ 0 \ge b_k \ge \frac{(n+1)}{n-1} b_k  >  b_{k+1}.$$
It follows that $b_{k+1}$ satisfies the same assumptions as $b_k$. So, by iteration,
$$0 \ge b_k > b_{k+1}>\dots>b_{n-2} > b_{n-1},$$
which contradicts Lemma~\ref{lem: b_{n-2}}.  It follows that $b_k > 0$ for all $0 \le k \le n-1$.

Next, let us prove that $(b_k)_{k=0}^{n-1}$ is strictly  decreasing. By Lemma~\ref{lem: b_{n-2}} and positivity,
$$  b_{n-2} = \dfrac{(n+1)}{n-1} b_{n-1} > b_{n-1} > 0.$$
Hence, if $(b_k)_{k=0}^{n-1}$ is not  strictly  decreasing, then there exists  $k$
 satisfying $1 \le k \le n-2$ such that $b_k > b_{k+1}$ and $b_k \ge b_{k-1}$. But then by \eqref{eq: convexcombo}
$$0 < b_k <  \frac{(n+1)}{n-1} b_k \le \max(b_{k-1}, b_{k+1}),$$
which is a contradiction.
\end{proof}\begin{corollary} \label{lem: c_kpos}
$a_k < 0$ for $1 \le k \le n-1$ and $c_k>0$ for $0 \le k \le n-2$.
\end{corollary} \begin{proof} By Lemmas~\ref{lem: a_k} and \ref{lem: pos+dec},
$$a_k = \frac{1}{2}(b_k - b_{k-1}) < 0,$$
and, by Lemma~\ref{lem: a_kc_k}, $c_k = -a_{k+1} > 0$.
\end{proof}

\subsubsection{Calculation of $\|Q_n\|_1$}
\begin{lemma} \label{lem: F(n)G(n)} $ \|Q_n\|_1 = F(n) + 2G(n)$, where
$$ F(n)  = \sum_{k=0}^{n-1} \binom{n-1}{k} b_k, \quad G(n) =  \sum_{k=0}^{n-2} (n-1-k) \binom{n-1}{k} c_k.$$
\end{lemma} \begin{proof}  Since the graph  automorphism  group of $\mathbb{Z}_2^n$ is edge-transitive, the distribution of $Q_n(e)$ is independent of $e \in E_n$. Hence
$$ \|Q_n\|_1 = \max_{e \in E_n} \|Q_n(e)\|_1 = \|Q_n(e_0)\|_1.$$
There are  $\binom{n-1}{k}$ edges of the form $\overrightarrow{v_k w_{k+1}}$, $(n-1-k) \binom{n-1}{k}$  of  form $\overrightarrow{v_k v_{k+1}}$,
and $(n-k)\binom{n-1}{k-1}$ of form $\overrightarrow{w_k w_{k+1}}$. Hence \begin{align*}
\|Q_n(e_0)\|_1 &= \sum_{k=0}^{n-1} \binom{n-1}{k} |b_k| + \sum_{k=0}^{n-2} (n-1-k)\binom{n-1}{k} |c_k|\\ &+  \sum_{k=0}^{n-2} (n-k)\binom{n-1}{k-1} |a_k|\\ &=  \sum_{k=0}^{n-1} \binom{n-1}{k} |b_k| +2 \sum_{k=0}^{n-2} (n-1-k)\binom{n-1}{k} |c_k|\\
\intertext{(by Lemma~\ref{lem: a_kc_k})}
&=  \sum_{k=0}^{n-1} \binom{n-1}{k} b_k +2 \sum_{k=0}^{n-2} (n-1-k)\binom{n-1}{k} c_k
\end{align*} by positivity of $b_k$ (Lemma~\ref{lem: pos+dec}) and $c_k$ (Lemma~\ref{lem: c_kpos}).
\end{proof} \begin{lemma} \label{lem: F(n)} $F(n)=1$.
\end{lemma}  \begin{proof}
 Let $x$ be the sum over all edge vectors of the form $\overrightarrow{v_k w_{k+1}}$ as $k$ ranges from $0$ to $n-1$. Note that $x$ is the sum over all $1$-parallel edge vectors. Hence $x$ is orthogonal to every small square cycle vector. So $x \in B(\mathbb{Z}_2^n)$, i.e., $Q_n(x)=x$.

Note that $F(n) = \sum_{u} \langle Q_n(e_0), u\rangle$, where $u$ ranges over all $1$-parallel edge vectors. Hence \begin{align*}
F(n) &= \sum_u  \langle Q_n(e_0), u\rangle \\
& = \sum_u  \langle e_0,Q_n( u) \rangle \\
&=  \langle e_0, Q_n( \sum_u  u) \rangle\\ &=  \langle e_0, Q_n(x) \rangle\\
&= \langle e_0, x \rangle\\ &=1.
\end{align*} \end{proof}
To state the next lemma, let us introduce some notation.  For $2 \le j \le n$:  \begin{enumerate} \item
let $x_j$ be the  sum over all $j$-parallel edge vectors of form $\overrightarrow{v_kv_{k+1}}$ (i.e., $\varepsilon_1 = 0$ at the initial and terminal vertices);
\item let $u_j$ be the sum over all $j$-parallel edge vectors of form $\overrightarrow{w_kw_{k+1}}$ (i.e., $\varepsilon_1 = 1$ at the initial and terminal vertices); \item let $y_j$ be the sum of all $1$-parallel edge vectors with $\varepsilon_j =0$ at the initial and terminal vertices;
\item let $z_j$ be the sum of all $1$-parallel edge vectors with $\varepsilon_j =1$ at the initial and terminal vertices.
\end{enumerate}
\begin{lemma} \label{lem: Q_n(x_j)} For $2 \le j \le n$, $Q_n(x_j) = \dfrac{3}{4} x_j + \dfrac{1}{4}(u_j + y_j - z_j)$. \end{lemma}
\begin{proof} It suffices to verify that \begin{equation} \label{eq: cutspace}
 \dfrac{3}{4} x_j + \dfrac{1}{4}(u_j + y_j - z_j) \in B(\mathbb{Z}_2^n) \end{equation}
and  \begin{equation} \label{eq: cyclespace} x_j - [\dfrac{3}{4} x_j + \dfrac{1}{4}(u_j + y_j - z_j)] = \frac{1}{4}(x_j - u_j - y_j + z_j) \in Z(\mathbb{Z}_2^n).
\end{equation}  To verify \eqref{eq: cutspace}, we check that  $\dfrac{3}{4} x_j + \dfrac{1}{4}(u_j + y_j - z_j)$ is orthogonal to every small square cycle vector.
It suffices to consider small square cycles with $1$-parallel or $j$-parallel (or both) edges as the support of $\dfrac{3}{4} x_j + \dfrac{1}{4}(u_j + y_j - z_j)$
is the union of the $1$-parallel and $j$-parallel edges:

Case 1: A small square cycle  $S$ with $1$-parallel and $i$-parallel edges, where $i \ne j$. Let $v_S$ be the corresponding cycle vector.

Case 1($\alpha$): If $\varepsilon_j =0$ on the edges of $S$, then $ x_j = u_j = z_j \equiv 0$ on the edges of $S$ and $y_j=1$ on the $1$-parallel pair of edges.  Hence
$$\langle  \dfrac{3}{4} x_j + \dfrac{1}{4}(u_j + y_j - z_j),  v_S \rangle = \frac{1}{4}\langle y_j, v_S \rangle = \frac{1}{4} - \frac{1}{4} = 0.$$

Case 1($\beta$):  If  If $\varepsilon_j =1$ on the edges of $S$, then $ x_j = u_j = y_j \equiv 0$ on the edges of $S$ and $z_j=1$ on the $1$-parallel pair of edges.  Hence
$$\langle  \dfrac{3}{4} x_j + \dfrac{1}{4}(u_j + y_j - z_j),  v_S \rangle = -\frac{1}{4}\langle z_j, v_S \rangle =- \frac{1}{4} + \frac{1}{4} = 0.$$

Case 2: A small square cycle  $S$ with $j$-parallel and $i$-parallel edges, where $i \ne 1$. Let $v_S$ be the corresponding cycle vector.

Case 2($\alpha$): If $\varepsilon_1 =0$ on the edges of $S$, then $  u_j =y_j=  z_j \equiv 0$ on the edges of $S$ and $x_j=1$ on the $j$-parallel pair of edges.  Hence
$$\langle  \dfrac{3}{4} x_j + \dfrac{1}{4}(u_j + y_j - z_j),  v_S \rangle = \frac{3}{4}\langle x_j, v_S \rangle = \frac{3}{4} - \frac{3}{4} = 0.$$

Case 2($\beta$):  If  If $\varepsilon_1 =1$ on the edges of $S$, then $ x_j  = y_j = z_j\equiv 0$ on the edges of $S$ and $u_j=1$ on the $j$-parallel pair of edges.  Hence
$$\langle  \dfrac{3}{4} x_j + \dfrac{1}{4}(u_j + y_j - z_j),  v_S \rangle = \frac{1}{4}\langle u_j, v_S \rangle = \frac{1}{4} - \frac{1}{4} = 0.$$

Cases 3: A small square cycle $S$ with $1$-parallel and $j$-parallel edges. First consider the $1$-parallel edges.
 On the $1$-parallel edge with $\varepsilon_j=0$, we have
$$ \dfrac{3}{4} x_j + \dfrac{1}{4}(u_j + y_j - z_j) = 0 +0 +\frac{1}{4} -0 = \frac{1}{4}$$
On the $1$-parallel edge with $\varepsilon_j=1$, we have
$$ \dfrac{3}{4} x_j + \dfrac{1}{4}(u_j + y_j - z_j) = 0 +0 +0 - \frac{1}{4}= -\frac{1}{4}$$
Now consider the $j$-parallel edges.  On the $j$-parallel edge with $\varepsilon_1=0$, we have
 $$ \dfrac{3}{4} x_j + \dfrac{1}{4}(u_j + y_j - z_j) = \frac{3}{4} +0 +0 -0 = \frac{3}{4}$$
 On the $j$-parallel edge with $\varepsilon_1=1$, we have
 $$ \dfrac{3}{4} x_j + \dfrac{1}{4}(u_j + y_j - z_j) = 0 +\frac{1}{4} +0 -0 = \frac{1}{4}.$$ Hence
$$\langle \dfrac{3}{4} x_j + \dfrac{1}{4}(u_j + y_j - z_j), v_S \rangle = \frac{1}{4}+ \frac{1}{4} -(-\frac{1}{4}) - \frac{3}{4} = 0.$$

To verify \eqref{eq: cyclespace}, we check that $x_j-u_j-y_j + z_j$ is orthogonal to each cut vector  $X(v)$ for $v = (\varepsilon_i)_{i=1}^n \in \mathbb{Z}_2^n$.

Case 1: $\varepsilon_1 = \varepsilon_j = 0$.

$$ \langle x_j-u_j-y_j + z_j, X(v) \rangle = 1\cdot 1 -0 -1\cdot1 +0 = 0.$$

Case 2: $\varepsilon_1 =0, \varepsilon_j = 1$.

$$ \langle x_j-u_j-y_j + z_j, X(v) \rangle = 1\cdot(-1)-0-0+ 1 \cdot 1 =0 = 0.$$

Case 3: $\varepsilon_1 =1, \varepsilon_j = 0$.

$$ \langle x_j-u_j-y_j + z_j, X(v) \rangle = 0 -1\cdot1 -1 \cdot (-1) +0 = 0.$$

Case 4: $\varepsilon_1 = \varepsilon_j = 1$.

$$ \langle x_j-u_j-y_j + z_j, X(v) \rangle =0- 1\cdot (-1) -0 + 1\cdot (-1)  = 0.$$
\end{proof}

\begin{lemma} \label{lem: G(n)} $G(n) = \dfrac{n-1}{4}$.
\end{lemma}

\begin{proof} Note that  $$\sum_{j=2}^n x_j = \sum_{\overrightarrow{v_k v_{k+1}}} e,$$ where the sum is taken over all edges $e$ of the form $\overrightarrow{v_k v_{k+1}}$.
Hence \begin{align*}
G(n) &=  \sum_{\overrightarrow{v_k v_{k+1}}} \langle Q_n(e_0),e \rangle \\
&= \langle Q_n(e_0),  \sum_{\overrightarrow{v_k v_{k+1}}} e \rangle\\
&= \langle Q_n(e_0), \sum_{j=2}^n x_j \rangle\\
&= \sum_{j=2}^n \langle e_0, Q_n(x_j) \rangle \\
& =  \sum_{j=2}^n \langle e_0,  \dfrac{3}{4} x_j + \dfrac{1}{4}(u_j + y_j - z_j) \rangle\\
\intertext{(by Lemma~\ref{lem: Q_n(x_j)})}
&= \sum_{j=2}^n (\frac{3}{4} \langle e_0,x_j\rangle + \frac{1}{4}[\langle e_0, u_j \rangle +\langle e_0, y_j \rangle-\langle e_0, z_j \rangle])\\
& = \sum_{j=2}^n[ 0 +0 + \frac{1}{4} - 0]\\
&= \frac{n-1}{4}.
\end{align*}
\end{proof} \begin{proof}[Proof of Theorem~\ref{thm: normoforthoghammingcube}]
By Lemmas~\ref{lem: F(n)G(n)}, \ref{lem: F(n)}, and  \ref{lem: G(n)},
$$ \|Q_n\|_1 = F(n) + 2G(n) = 1 + 2 \frac{(n-1)}{4} =  \frac{n+1}{2}.$$
\end{proof} \begin{corollary} Let $P_n$ be the orthogonal projection onto $Z(\mathbb{Z}_2^n)$. Then
$$\|P_n\|_1 = \frac{n+3}{2} - \frac{4}{n} + \frac{1}{n2^{n-2}}.$$
\end{corollary} \begin{proof} \begin{align*} \|P_n\|_1 &= \|I - Q_n\|_1\\
& = \sum_{e \in E_n} |\langle e_0 - Q_n(e_0), e \rangle|\\
&=  \sum_{e \in E_n} |\langle Q_n(e_0), e \rangle| +1 - 2\langle Q_n(e_0), e_0 \rangle\\
&= \|Q_n\|_1 +1 - 2b_0\\ &= \frac{n+1}{2} +1 - 2( \dfrac{2}{n} - \dfrac{1}{n 2^{n-1}})\\
\intertext{(by Lemma~\ref{lem: b_0} and Theorem~\ref{thm: normoforthoghammingcube})}
&=  \frac{n+3}{2} - \frac{4}{n} + \frac{1}{n2^{n-2}}.
\end{align*}
\end{proof}

\subsubsection{The projection constant of
$\operatorname{Lip}_0(\mathbb{Z}_2^n)$}

\begin{lemma} Let $G=(V,E)$ be a finite connected graph with edge space $E(G)$ and  automorphism group $\Gamma$.   Then $\hat{\phi}^* = \hat{\phi}^{-1}$
for all $\phi \in \Gamma$.
\end{lemma} \begin{proof} For all $x \in E(G)$,
$$ \hat{\phi}(x)= \sum_{e \in E} \langle x, \overrightarrow{e_-e_+}\rangle \overrightarrow{\phi(e_-)\phi(e_+)} = \sum_{e \in E}
 \langle x,  \overrightarrow{\phi^{-1}(e_-)\phi^{-1}(e_+)} \rangle \overrightarrow{e_-e_+}.$$
Hence, for all $x,y \in E(G)$, \begin{align*}
\langle  \hat{\phi}(x), y \rangle &=  \sum_{e \in E} \langle x,  \overrightarrow{\phi^{-1}(e_-)\phi^{-1}(e_+)} \rangle \langle y, \overrightarrow{e_-e_+}\rangle\\&=\sum_{e \in E} \langle x, \overrightarrow{e_-e_+} \rangle \langle y,  \overrightarrow{\phi(e_-)\phi(e_+)} \rangle. \end{align*} Thus,
\begin{align*}  \hat{\phi}^*(y) & = \sum_{e \in E}  \langle y,  \overrightarrow{\phi(e_-)\phi(e_+)} \rangle \overrightarrow{e_-e_+}\\
&= \sum_{e \in E} \langle y, \overrightarrow{e_-e_+}\rangle  \overrightarrow{\phi^{-1}(e_-)\phi^{-1}(e_+)} \\
&=  \hat{\phi}^{-1}(y).
\end{align*}
\end{proof}

\begin{corollary} \label{cor: T^*} If $T \colon E(G) \rightarrow  E(G)$ is $\Gamma$-invariant then $T^*$ is also $\Gamma$-invariant.
\end{corollary}

\begin{theorem} \label{thm: Hammingcubeprojconstant} $\lambda(\operatorname{Lip}_0(\mathbb{Z}_2^n))  = \dfrac{n+1}{2}.$
\end{theorem}

\begin{proof}  Let $Q$ be a $\Gamma_n$-invariant projection from $E(\mathbb{Z}_2^n)$ onto $B(\mathbb{Z}_2^n)$.
By Corollary~\ref{cor: T^*}, $I - Q^*$ is a $\Gamma_n$-invariant
projection onto $B(\mathbb{Z}_2^n)^\perp = Z(\mathbb{Z}_2^n)$.
Hence, by Corollary~\ref{cor: 23}, $I-Q^* = P_n$. So $Q = I-P_n =
Q_n$. So $Q_n$ is a minimal projection from $(E(\mathbb{Z}_2^n),
\|\cdot\|_\infty)$ onto $(B(\mathbb{Z}_2^n),\|\cdot\|_\infty)$.
Since  $(B(\mathbb{Z}_2^n),\|\cdot\|_\infty)$ is isometrically
isomorphic to  $\operatorname{Lip}_0(\mathbb{Z}_2^n)$,
Theorem~\ref{thm: Hammingcubeprojconstant} gives
$$ \lambda(\operatorname{Lip}_0(\mathbb{Z}_2^n))  = \|Q_n\|_\infty = \|Q_n\|_1 = \frac{n+1}{2}.$$
\end{proof}

\begin{corollary}\label{C:BM} $$\frac{n+1}{2} \le d_{\BM}(\tc(\mathbb{Z}_2^n), \ell_1^{2^n -1})\le 2n.$$ \end{corollary}

\begin{proof} Suppose that $f \colon \mathbb{Z}_2^n \rightarrow \mathbb{R}$ satisfies $f((0,0,\dots,0))=0$. Then
$$\dfrac{\|f\|_\infty}{n}
\le \sup_{x\ne y \in \mathbb{Z}_2^n} \dfrac{|f(x)-f(y)|}{d(x,y)}\\
\le 2\|f\|_\infty. $$ Hence $d_{\BM}(\operatorname{Lip}_0(\mathbb{Z}_2^n), \ell_\infty^{2^n -1})\le 2n.$ On the other hand,
$$d_{\BM}(\operatorname{Lip}_0(\mathbb{Z}_2^n), \ell_\infty^{2^n -1})\ge \lambda(\operatorname{Lip}_0(\mathbb{Z}_2^n)) = \frac{n+1}{2}.$$
By duality,  $d_{\BM}(\operatorname{Lip}_0(\mathbb{Z}_2^n), \ell_\infty^{2^n -1}) =d_{\BM}(\tc(\mathbb{Z}_2^n), \ell_1^{2^n -1})$.
\end{proof}

\begin{remark} It is known that  $cn\le c_1(\tc(\mathbb{Z}_2^n))\le
Cn$ for some absolute constants $0<c\le C<\infty$. The left-hand
side inequality is due to \cite[Corollary 2]{KN06}, the right-hand
side inequality is due to \cite{Cha02,IT03,FRT04}. The left-hand
side of the inequality implies a weaker version of the left-hand
side inequality in Corollary \ref{C:BM}.
\end{remark}

\begin{acknowledgement} We thank Kevin Ford for numerical results related to the recurrence relation in Lemma~\ref{lem: recurrence}.
The first author was supported by Simons Foundation Collaboration
Grant No. 849142. The second author was supported by Simons
Foundation Collaboration Grant No. 636954.  The first and second
authors would like to thank the Isaac Newton Institute, Cambridge,
for support and hospitality during the programme Approximation,
Sampling and Compression in Data Science, where work on this paper
was undertaken. This work was supported by ESPRC Grant no
EP/K032208/1. The third-named author gratefully acknowledges the
support by the National Science Foundation grant NSF DMS-1953773.
\end{acknowledgement}

\end{document}